\documentclass[11pt]{amsart}
\usepackage[utf8]{inputenc}
\usepackage{kantlipsum}
\usepackage[pagebackref=true, colorlinks=true,linkcolor=blue, citecolor={red}]{hyperref}
\usepackage{lmodern}\usepackage[T1]{fontenc}
\usepackage{amsmath,amssymb,amsfonts,euscript,graphicx,hyperref,indentfirst}
\usepackage{enumerate,linegoal}
\usepackage[all]{xy}
\usepackage{tikz-cd}
\usepackage{stmaryrd}
\usepackage{relsize,exscale}
\usepackage{bigints}
\usepackage{cancel}
\usepackage{accents}
\usepackage{array}
\usepackage{pbox}
\usepackage{tikz}
\usetikzlibrary{arrows}
\usepackage{multirow}
\newcolumntype{P}[1]{>{\centering\arraybackslash}p{#1}}

\calclayout

\newtheorem*{corollary*}{Corollary}

\newtheorem{theorem}{Theorem}[section]
\newtheorem{lemma}[theorem]{Lemma}
\newtheorem{proposition}[theorem]{Proposition}
\newtheorem{corollary}[theorem]{Corollary}

\newtheorem{definition-theorem}[theorem]{Definition-Theorem}

\theoremstyle{definition}

\newtheorem{example}[theorem]{Example}

\theoremstyle{remark}
\newtheorem{remark}[theorem]{Remark}
\makeatletter
\def\l@subsection{\@tocline{2}{0pt}{2.5pc}{5pc}{}} 
\makeatother

\numberwithin{equation}{section}

\title{On the rank of the Thurston pullback map}
\author{Khashayar Filom}

\address{Khashayar Filom, Department of Mathematics, University of Michigan;
 530 Church St, Ann Arbor, MI 48109,
USA}
\email{filom@umich.edu}

\date{}
\begin{document}
\vspace*{-1cm}
\maketitle
{\centering\footnotesize  Dedicated to Professor Mehrdad Shahshahani.\par}
\begin{abstract}
Under some mild assumptions, an orientation-preserving branched covering map of marked $2$-spheres induces a pullback map between the corresponding  Teichm\"uller spaces. 
By analyzing the associated pushforward operator acting on integrable quadratic differentials,  we obtain a global lower bound on the rank of the derivative of the pullback map in terms of the action of the cover on the marked points. In the dynamical context, the two sets of marked points in the target and source coincide with the postcritical set. Investigating the resulting pullback map is the central part of Thurston's topological characterization of postcritically finite rational maps.  
Postcritically finite maps with constant pullback have been studied by various authors. In that direction, our approach provides upper bounds on the size of the postcritical set of a map with constant pullback, and shows that the postcritical dynamics is highly restricted.  
\end{abstract}

\section{Introduction}
\subsection{Motivation and summary of results}
The goal of this paper is to provide lower bounds for the rank of certain maps between Teichmüller spaces modeled on marked $2$-spheres, and also to study cases where these maps are constant. These Teichmüller maps are associated with branched covering maps between spheres through a construction that appears in the seminal work of W. Thurston on topological characterization of rational maps (see the paper \cite{MR1251582} by Douady and Hubbard for an account of this theory).  To elaborate, consider an orientation-preserving branched cover 
$f:\left(S^2,A\right)\rightarrow\left(S^2,B\right)$. The cover is called to be \textit{admissible} if 
\begin{itemize}
\item $A$ and  $B$ are finite subsets of the sphere each with at least three points; 
\item $f(A)\subseteq B$;
\item $B$ contains the critical values of $f$.  
\end{itemize}
The corresponding \textit{Thurston pullback map} is denoted by 
$\sigma_f:\mathcal{T}\left(S^2,B\right)\rightarrow\mathcal{T}\left(S^2,A\right)$ and satisfies the following: 
if the map $\sigma_f$ takes a Teichm\"uller point $\tau$ to another Teichm\"uller point $\tau'$, then any two homeomorphisms $\phi$ and $\phi'$ representing $\tau$ and $\tau'$ fit in a commutative diagram 
\begin{equation}\label{diagram}
\xymatrixcolsep{4pc}\xymatrix{\left(S^2,A\right)\ar[r]^{\phi'}\ar[d]_f &\left(\hat{\Bbb{C}},\phi'(A)\right)\ar[d]^g\\ 
\left(S^2,B\right)\ar[r]_\phi &\left(\hat{\Bbb{C}},\phi(B)\right)}
\end{equation}
where  $g:\hat{\Bbb{C}}\rightarrow\hat{\Bbb{C}}$ is a suitable rational map (see \cite{MR2508269,MR3107522} for details of this construction).

\begin{theorem}\label{rank bound}
With the notation as above, at any point $\tau$ of the Teichm\"uller space, the rank of the derivative ${\rm{D}}\sigma_f(\tau)$ admits the lower bound 
\begin{equation}\label{inequality1}
{\rm{rank}}\, {\rm{D}}\sigma_f(\tau)\geq 
\min\left\{\ell_1+\ell_2,|A|-3\right\}
\end{equation}
where 
\begin{equation}\label{ell1}
\ell_1:=\left|\{b\in B\mid b \text{ is not a critical value of } f \text{ and } |A\cap f^{-1}(b)|=1\}\right|,
\end{equation}
and 
\begin{equation}\label{ell2}
\ell_2:=\left|\left\{c\in S^2-A \,\Bigg|\,
\begin{array}{l}
c \text{ is  a } \text{simple critical point of  } f \text{ with } f^{-1}(f(c))\cap A=\emptyset\\
\text{and no  other point of } f^{-1}(f(c))\text{ is critical}. 
\end{array}
\right\}\right|.
\end{equation}
\end{theorem}
\noindent
Recall that  a critical point $c$ is called \textit{simple} (nondegenerate) if the multiplicity of $f$ at it is two (i.e. $f$ is two-to-one near $c$).  Notice that  the sets from \eqref{ell1} and \eqref{ell2} are in a sense complementary: in the former, $b$ is not a critical value of $f$ while in the latter, $f(c)$ is a critical value. 

The lower bound $\ell_1$ for the rank of the derivative clearly puts restrictions on admissible covers whose pullback maps become of rank zero at a point. 
\begin{corollary}\label{ell1 corollary}
With the notation as above, if the rank of the pullback map $\sigma_f$ is zero at a point and $|A|\geq 4$, then for any $a\in A$ either $f(a)$ is a critical value or there exists another element $a'$ of $A$ for which $f(a)=f(a')$.
\end{corollary}
\noindent
Also in the case of the derivative vanishing at a point, we derive  an upper bound in terms of the degree for the number of marked points in the domain of $f$. 
\begin{theorem}\label{degree bound}
With the notation as above and denoting the degree of $f$ by $d$, the derivative of $\sigma_f$ does not vanish at any point of $\mathcal{T}\left(S^2,B\right)$ unless 
$|A|\leq d+2$. 
Moreover, if there exists a critical value $v$ of $f$ such that $f^{-1}(v)\cap A=\emptyset$ and 
all critical points in $f^{-1}(v)$ are simple,  then the inequality can be improved to
$$
|A|\leq\#\left(\text{critical points in } f^{-1}(v)\right)+2.
$$
\end{theorem}
\noindent
The results above on the rank of $\sigma_f$ will be established in \S3 via studying its cotangent map ${\rm{D}}^*\sigma_f$ which can be identified with the pushforward action of $f$ on quadratic differentials. 

In the dynamical setting, one takes  $f:S^2\rightarrow S^2$ to be a \textit{Thurston map}; that is,  an orientation-preserving topological branched cover whose critical points (i.e. the points in a vicinity of which $f$ fails to be injective) are eventually periodic. Denoting the sets of critical points and critical values of $f$ by ${\rm{C}}_f$ and ${\rm{V}}_f:=f({\rm{C}}_f)$, the finite sets $A$ and $B$ are taken to be the \textit{postcritical set} ${\rm{P}}_f$ defined as 
$\bigcup_{n\geq 0}f^{\circ n}({\rm{V}}_f)$. 
Thurston's characterization states that, except for a handful of well-studied Euclidean examples, $f$ is combinatorially equivalent to a \textit{postcritically finite} (PCF) rational map -- a family of maps which is of great interest in complex dynamics -- if and only if it does not admit any \textit{obstructing multicurve} (see \cite{MR1251582}).
This is established by showing that rational maps equivalent to $f$ correspond to fixed points of the associated pullback map 
$\sigma_f:\mathcal{T}\left(S^2,{\rm{P}}_f\right)\rightarrow\mathcal{T}\left(S^2,{\rm{P}}_f\right)$, and then studying the fixed points of $\sigma_f$. In this dynamical setting, Corollary \ref{ell1 corollary} and  Theorem \ref{degree bound} mentioned above, together with a careful investigation of the combinatorics of the map $f\restriction_{{\rm{P}}_f}:{\rm{P}}_f\rightarrow {\rm{P}}_f$  in Proposition \ref{combinatorics}, yield the following:
\begin{corollary}\label{main corollary}
Let $f:S^2\rightarrow S^2$ be a Thurston map of degree $d$ with $|{\rm{P}}_f|\geq 4$ for which the derivative of the pullback map $\sigma_f$ vanishes at a point. 
Then  one has $|f^{-1}(p)\cap {\rm{P}}_f|\geq 2$ for any $p\in {\rm{P}}_f-{\rm{V}}_f$. Moreover:
\begin{equation}\label{inequality2}
|{\rm{P}}_f|\leq \min\left\{2\,|{\rm{V}}_f|, d+2\right\}.
\end{equation}
\end{corollary}
\noindent
A very special case of having a fixed point is when 
$\sigma_f:\mathcal{T}\left(S^2,{\rm{P}}_f\right)\rightarrow\mathcal{T}\left(S^2,{\rm{P}}_f\right)$
is constant. Thurston maps with constant pullback have been of interest in the literature \cite{MR2508269,MR2958932,MRCui,MR3692494,MR3456156,MRArturo,MR3994789}. The corollary above restricts both the size and the dynamics on the postcritical set for such maps. The lower bound $\ell_2$ for the rank of the pullback can be used to exclude certain Thurston maps from this family, e.g. maps for which the subset of $A=B={\rm{P}}_f$  appearing in \eqref{ell2} is nonvacuous. 
\begin{corollary}\label{ell2 corollary}
Let $f$ be a Thurston map with $|{\rm{P}}_f|\geq 4$ admitting a strictly preperiodic critical value $v$ which is not in the forward orbit of any other critical value, i.e. $v\notin f({\rm{P}}_f)$. Suppose there is only one critical point in $f^{-1}(v)$, and that critical point is simple. Then the pullback map $\sigma_f$ is not constant.   
\end{corollary}
\noindent
By a trick due to McMullen (\cite{MR2508269}), one can obtain examples of maps with constant Thurston pullback through composing maps of lower degrees. An interesting example with constant pullback which cannot be obtained in this manner is a quartic rational map with three critical values introduced in \cite[Appendix D]{MRArturo}.  We shall utilize our results to investigate if the pullback can be constant for Thurston maps
with at most three critical values, or for Thurston maps of degree three. 
It turns out that if $f$ is \textit{bicritical} (i.e. has only two critical points) or \textit{cubic} (i.e. $\deg f=3$), 
the pullback map $\sigma_f:\mathcal{T}\left(S^2,{\rm{P}}_f\right)\rightarrow\mathcal{T}\left(S^2,{\rm{P}}_f\right)$
is not constant
except  in the dull case of $|{\rm{P}}_f|=3$ where $\dim\mathcal{T}\left(S^2,{\rm{P}}_f\right)=0$.
\begin{theorem}\label{bicritical or cubic}
Let $f$ be a Thurston map with at least four postcritical points. 
The pullback map $\sigma_f$ is not constant if $f$ is bicritical or cubic. 
\end{theorem}
\noindent
For cubic Thurston maps, this result was previously proved by Koch through an extensive study of Hurwitz classes of ramified covers of degree three \cite{MRKoch}.

\subsection{Outline.}
Section 2 is devoted to the background material: \S2.1 briefly discusses Thurston's work in the dynamical context which is concerned with Thurston maps $f:\left(S^2,{\rm{P}}_f\right)\rightarrow\left(S^2,{\rm{P}}_f\right)$,
and also the more general nondynamical setting of admissible branched covers 
$f:\left(S^2,A\right)\rightarrow\left(S^2,B\right)$; \S2.2 reviews the literature on branched covers whose associated Thurston pullback map is constant. Our main results, Theorems \ref{rank bound} and \ref{degree bound}, are proved in respectively \S3.1 and \S3.2 through a careful analysis of the action of the pushforward operator on integrable quadratic differentials. The next subsection, \S3.3, establishes numerous corollaries to these theorems, in particular, Corollary \ref{main corollary}. Results from \S3 are invoked extensively in the final section, \S4, to investigate Thurston maps with constant pullback. The section begins by discussing Thurston maps with three critical values and bicritical Thurston maps in \S4.1 and \S4.2, and culminates with a proof of Theorem \ref{bicritical or cubic} in \S4.3.

\subsection{Notation and terminology}
Every branched cover is considered to be orientation-preserving and of degree $d\geq 2$. We often work with ramified covering maps $\left(S^2,A\right)\rightarrow\left(S^2,B\right)$ that are admissible, meaning that each of $A$ and $B$ is of size at least three, and $B$ includes the critical values of the cover. 
\begin{itemize}
\item $f^{\circ n}$: the $n^{\rm{th}}$ iterate of $f$;
\item $X\hookrightarrow Y$ and  $X\stackrel{inc}{\hookrightarrow}Y$: an injection from $X$ to $Y$ or the inclusion map if
$X\subseteq Y$;
\item $S^2$ and $\hat{\Bbb{C}}$: the $2$-sphere and the Riemann sphere $\hat{\Bbb{C}}=\Bbb{C}\cup\{\infty\}$;
\item $T(z_1,z_2,z_3,z_4)$: the cross-ratio of four pairwise distinct points $z_1,z_2,z_3,z_4$ of  $\hat{\Bbb{C}}$ defined as 
$T(z_1,z_2,z_3,z_4)=\frac{(z_1-z_2)(z_3-z_4)}{(z_1-z_3)(z_2-z_4)}$;
\item ${\rm{PSL}}_2(\Bbb{C})$: the group of Möbius transformations; 
\item $Q^1\left(\hat{\Bbb{C}}-C\right)$: the space of integrable holomorphic quadratic differentials on $\hat{\Bbb{C}}-C$, i.e. the space of those quadratic differentials that at worst have simple poles at the points of $C$;  
\item ${\rm{T}}_xM$ (${\rm{T}}^*_xM$): the tangent (resp. cotangent space) of a smooth manifold $M$ at a point $x$;
\item  ${\rm{D}}F(x)$ (${\rm{D}}^*F(x)$): the derivative ${\rm{T}}_xM\rightarrow{\rm{T}}_{F(x)}N$ (resp. the coderivative ${\rm{T}}^*_{F(x)}N\rightarrow{\rm{T}}^*_xM$) of a morphism $F:M\rightarrow N$ of smooth manifolds at a point $x\in M$; 
\item ${\rm{C}}_f$ and ${\rm{V}}_f$: the sets of critical points and critical values (branch points) of a branched cover $f$ respectively ($f({\rm{C}}_f)={\rm{V}}_f$);
\item ${\rm{P}}_f$: the postcritical set of a Thurston map $f:S^2\rightarrow S^2$ (a branched cover with eventually periodic critical values) defined as $\bigcup_{n\geq 0}f^{\circ n}({\rm{V}}_f)$;
\item $\mathcal{G}_f$: the functional graph capturing the postcritical dynamics 
$f\restriction_{{\rm{P}}_f}:{\rm{P}}_f\rightarrow{\rm{P}}_f$ of a Thurston map $f$; the vertices represent elements of ${\rm{P}}_f$ and there is an edge from $p$ to $q$ if $f(p)=q$;
\item ${\rm{deg}}^+(p)$ (${\rm{deg}}^-(p)$): the outdegree (resp. indegree) of a vertex $p$ of a directed graph; 
\item ${\rm{Deck}}(f)$: the group of deck transformations of a branched cover $f$;
\item ${\rm{mult}}_{p}(f)$: the multiplicity with which $f$ attains  the value $f(p)$ at $p$;
\item $\mathcal{T}\left(S^2,C\right)$: the Teichm\"uller space modeled on the marked sphere $\left(S^2,C\right)$;
\item $\mathcal{M}_{0,C}$: the moduli space of genus zero Riemann surfaces marked by $\left(S^2,C\right)$; that is, the space of injections $i:C\hookrightarrow\hat{\Bbb{C}}$ modulo postcomposition by Möbius transformations; 
\item $\sigma_f$: the Thurston pullback map 
$\mathcal{T}\left(S^2,B\right)\rightarrow\mathcal{T}\left(S^2,A\right)$ 
(resp. $\mathcal{T}\left(S^2,{\rm{P}}_f\right)\rightarrow\mathcal{T}\left(S^2,{\rm{P}}_f\right)$)
induced by $f:\left(S^2,A\right)\rightarrow\left(S^2,B\right)$ (resp. by a Thurston map 
$\left(S^2,{\rm{P}}_f\right)\rightarrow\left(S^2,{\rm{P}}_f\right)$);
\item $\mathcal{W}_f$ and $\mathcal{R}_f$: the Hurwitz space and the correspondence associated with a branched cover $f:\left(S^2,A\right)\rightarrow\left(S^2,B\right)$ (cf. \S2.1);
\item $\ell_1$ and $\ell_2$: quantities associated with a branched cover $f:\left(S^2,A\right)\rightarrow\left(S^2,B\right)$ as in \eqref{ell1} and \eqref{ell2};
\item $P\oplus Q$: the sum of points $P$ and $Q$ of an elliptic curve.

\end{itemize}

\section{Preliminaries and known results}
This section begins with a brief discussion on Thurston's theorem on topological characterization of rational maps which is contingent on a careful analysis of the corresponding pullback map between  Teichm\"uller spaces. The second subsection reviews some known results on branched covers with constant pullback and compares them with our results in \S4.

\subsection{The Thurston pullback map}
Let $f:S^2\rightarrow S^2$ be a Thurston map, an orientation-preserving branched cover whose critical values are preperiodic. 
We denote its set of critical values by ${\rm{V}}_f$ and its postcritical set $\bigcup_{n\geq 0}f^{\circ n}({\rm{V}}_f)$ by 
${\rm{P}}_f$. 
Given an orientation-preserving homeomorphism $\phi:S^2\rightarrow\hat{\Bbb{C}}$, one can pullback the complex structure of $\hat{\Bbb{C}}$ first via $\phi$ and then via the branched cover $f:S^2\rightarrow S^2$. The resulting complex structure on the domain of $f$ can be uniformized by a suitable orientation-preserving homeomorphism $\phi':S^2\rightarrow\hat{\Bbb{C}}$. We thus arrive at the commutative diagram
\begin{equation}\label{diagram1}
\xymatrixcolsep{4pc}\xymatrix{\left(S^2,{\rm{P}}_f\right)\ar[r]^{\phi'}\ar[d]_f &\left(\hat{\Bbb{C}},\phi'\left({\rm{P}}_f\right)\right)\ar[d]^g\\ 
\left(S^2,{\rm{P}}_f\right)\ar[r]_\phi &\left(\hat{\Bbb{C}},\phi\left({\rm{P}}_f\right)\right)}
\end{equation}
in which $g:\hat{\Bbb{C}}\rightarrow\hat{\Bbb{C}}$ is holomorphic, i.e. a rational map. 
Perturbing the marked homeomorphism $\phi:\left(S^2,{\rm{P}}_f\right)\rightarrow\left(\hat{\Bbb{C}},\phi\left({\rm{P}}_f\right)\right)$ within its isotopy class does not change the isotopy class of  the homeomorphism
$\phi':\left(S^2,{\rm{P}}_f\right)\rightarrow\left(\hat{\Bbb{C}},\phi'\left({\rm{P}}_f\right)\right)$
due to the homotopy lifting property. Therefore, by the construction above, the branched cover 
$f:\left(S^2,{\rm{P}}_f\right)\rightarrow\left(S^2,{\rm{P}}_f\right)$
induces a holomorphic map 
$\sigma_f:\mathcal{T}\left(S^2,{\rm{P}}_f\right)\rightarrow\mathcal{T}\left(S^2,{\rm{P}}_f\right)$
through the assignment $[\phi]\mapsto[\phi']$.
We call this the Thurston pullback map associated with the cover. Thurston's theorem elucidates when $f$ is \textit{Thurston equivalent} to a rational map, meaning that $f$ fits in a diagram such as \eqref{diagram1} in which homeomorphisms
 $\phi,\phi':S^2\rightarrow\hat{\Bbb{C}}$ agree on ${\rm{P}}_f$ and are isotopic through homeomorphisms agreeing on ${\rm{P}}_f$. 
This requires the rational map $g$ from \eqref{diagram1} to be postcritically finite; and in terms of the pullback map, it amounts to $[\phi]=[\phi']$ being a fixed point of $\sigma_f$. 
\begin{theorem}[\cite{MR1251582}]
Let $f$ be a Thurston map with hyperbolic orbifold. Then $f$ is equivalent to a PCF rational map
if and only if for any  $f$-stable multicurve $\Gamma$ the Thurston eigenvalue $\lambda_\Gamma$ is smaller than $1$.
Moreover, if $f$ is equivalent to a PCF rational map, such a map is unique up to Möbius conjugation (Thurston rigidity).   
\end{theorem}
\noindent
Here, an \textit{$f$-stable multicurve} $\Gamma$ is a finite collection $\{\gamma_1,\dots,\gamma_n\}$ of disjoint nonhomotopic simple closed curves in $S^2-{\rm{P}}_f$ which are not peripheral (i.e. each component $S^2-\gamma_i$ contains at least two points of ${\rm{P}}_f$); and  every nonperipheral component of $f^{-1}(\gamma_i)$ is  homotopic to an element of $\Gamma$ relative to  ${\rm{P}}_f$.
The \textit{Thurston eigenvalue} 
$\lambda_\Gamma$ then denotes the leading eigenvalue of a linear map  
$f_\Gamma:\Bbb{R}^\Gamma\rightarrow\Bbb{R}^\Gamma$ defined as 
$$f_\Gamma(\gamma)=\sum_{\delta\in\Gamma}\left(\sum_{\eta\text{ a component of } f^{-1}(\gamma) \text{ homotopic to } \delta \text{ rel. } P_f }
\left(\frac{1}{\deg(f\restriction_\eta:\eta\rightarrow\gamma)}\right)\right)\delta.
$$
Finally, the hyperbolicity condition for the orbifold of $f$ is to rule out certain well-known Euclidean examples such as \textit{Lattès maps} (see \cite[\S9]{MR1251582} for nonhyperbolic cases).   

In a more general point of view adapted in \cite{MR2508269,MR3107522}, one divorces the domain and codomain and considers orientation-preserving branched covers $f:\left(S^2,A\right)\rightarrow\left(S^2,B\right)$ instead of Thurston maps. 
As mentioned before, we always take $f$ to be admissible; namely, $B\supseteq {\rm{V}}_f$ (so that the homotopy lifting property holds), and finite sets $A$ and $B$ each have at least three elements (so that $S^2-A$, $S^2-B$ are hyperbolic).   
Now the commutative diagram \eqref{diagram1} should be replaced with \eqref{diagram}, and the pullback map would be in the form of $\sigma_f:\mathcal{T}\left(S^2,B\right)\rightarrow\mathcal{T}\left(S^2,A\right)$.

Various aspects of pullback maps have been studied in  \cite{MR4100123,MREpstein,MR3668372}.
In general, it is hard to describe a Thurston pullback map $\sigma_f$ since it is a map between transcendental, non-algebraic spaces. The article \cite{MR2508269} provides examples of Thurston maps $f$ for which $\sigma_f$ is a Galois ramified cover or an unramified cover onto a dense open subset. This is based on the work of Koch which shows that  the transcendental map $\sigma_f$ descends to an algebraic correspondence on the level of moduli spaces \cite{MR3107522}. To elaborate, let
$f:\left(S^2,A\right)\rightarrow\left(S^2,B\right)$ 
be an admissible cover. Then $\sigma_f$ fits in a commutative diagram such as 
\begin{equation}\label{correspondence}
\xymatrix{\mathcal{T}\left(S^2,B\right)\ar[dd]_{\pi_B} \ar[rd]^{\omega_f}  \ar[rr]^{\sigma_f}& &\mathcal{T}\left(S^2,A\right)\ar[dd]^{\pi_A}\\
&\mathcal{W}_f \ar[ld]^{\rho_B}\ar[rd]_{\rho_A}&\\
\mathcal{M}_{0,B}  & &\mathcal{M}_{0,A}}
\end{equation}
where
\begin{itemize}
\item $\mathcal{W}_f$ is the \textit{Hurwitz space} of $f$; each point of $\mathcal{W}_f$ is a rational map 
$g:\hat{\Bbb{C}}\rightarrow\hat{\Bbb{C}}$ which (as a marked branched cover) is \textit{Hurwitz equivalent} to 
$f:S^2\rightarrow S^2$ (i.e. a diagram such as \eqref{diagram} commutes), and is considered modulo conformal changes of coordinates in domain and codomain. 
\item A point of $\mathcal{W}_f$ determined by a rational map $g$ Hurwitz equivalent to $f$ as in \eqref{diagram} is sent by $\rho_A$ and $\rho_B$ to classes of injections $A\hookrightarrow\hat{\Bbb{C}}$ and $B\hookrightarrow\hat{\Bbb{C}}$ determined by the rows of \eqref{diagram}.   
\item $\omega_f:\mathcal{T}\left(S^2,B\right)\rightarrow\mathcal{W}_f$ and 
$\rho_B:\mathcal{W}_f\rightarrow\mathcal{M}_{0,B}$
are unramified covering maps (hence surjective) and the latter is finite \cite[Theorem 2.6]{MR3107522}.
\end{itemize}
All in all, $\mathcal{R}_f:=\rho_A\circ\rho_B^{-1}:\mathcal{M}_{0,B}\rightrightarrows\mathcal{M}_{0,A}$ is a multi-valued algebraic map -- a correspondence. One interesting case, discussed in \cite[\S 3]{MR3107522}, is when $\rho_A$ is injective and $\sigma_f$ induces a moduli space map 
$\mathcal{R}_f^{-1}:\mathcal{M}_{0,A}\dashrightarrow\mathcal{M}_{0,B}$ 
in the opposite direction. The other extreme is when $\rho_A$ is constant which is equivalent to the constancy of $\sigma_f$. 
The article  \cite{MR2508269} also provides examples of Thurston maps $f$ (where $A=B={\rm{P}}_f$) for which $\sigma_f$ is constant.  We shall review some known results regarding such maps in the next subsection.

\subsection{The case of constant pullback}
A characterization of covers $f:\left(S^2,A\right)\rightarrow\left(S^2,B\right)$
with $\sigma_f$ constant in terms of various objects associated with $f$ can be found in \cite[Theorem 5.1]{MR3456156} (the result originally appeared in \cite{MR2508269}). In particular, $\sigma_f$ is constant if and only if for all simple closed curves 
$\gamma\subset S^2-B$ any component of $f^{-1}(\gamma)\subset S^2-A$ is peripheral.  

Alternatively, one can construct Thurston maps of low rank by a composition trick due to McMullen (\cite{MR2508269}): if 
$f:\left(S^2,A\right)\rightarrow\left(S^2,B\right)$ factors through $(S^2,C)$, then the pullback map
 $\sigma_f:\mathcal{T}\left(S^2,B\right)\rightarrow\mathcal{T}\left(S^2,A\right)$ 
factors through $\mathcal{T}\left(S^2,C\right)$, hence ${\rm{rank}}\, {\rm{D}}\sigma_f\leq |C|-3$. 
In particular, $\sigma_f$ is constant if $|C|=3$. It is possible to find such examples in the dynamical context too; see \cite{MR2508269} for details. Nevertheless, there is an example of a Thurston map of degree four with constant pullback which cannot be obtained this way \cite[Appendix D]{MRArturo}. In \S4.1, we shall discuss this example and some possible generalizations of it. It is an open question if the pullback can be constant for a Thurston map of a prime degree \cite{MR4472056}.

The question of constancy of the pullback can be furthermore considered for \textit{nearly Euclidean Thurston} (NET) maps, which is a very well-studied family \cite{MR2958932,MR3692494}.   A Thurston map $f$ is a NET map if $|{\rm{P}}_f|=4$ and every critical point of $f$ is simple (nondegenerate). See \cite[Theorem 5]{MR3692494} for a characterization of NET maps with constant pullback. It is known that there does exist a NET map of degree $d$ with constant pullback if 
$d>2$ is divisible by $2$ or $9$ (\cite[Theorem 10.8]{MR2958932}) whereas there is no such an example if 
$d$ is squarefree and odd (\cite[Theorem 10.12]{MR2958932}), or if $d$ is the square or the cube of an odd prime (with certain exceptions), or has no prime divisor less than $13$ (\cite[Corollaries 4.6 and 4.9 and 4.10]{MR3994789}). 
In particular, it is known that the pullback cannot be constant for any NET map of degree two (more generally, for any Thurston map of degree two; cf. \cite[Appendix B]{MRArturo}) or three. Theorem \ref{bicritical or cubic} greatly generalizes this by considering all bicritical or cubic Thurston maps (NET or otherwise). The case of cubic Thurston maps can alternatively be addressed by an exhaustive examination of all combinatorial possibilities \cite{MRKoch}.

\section{Investigating the pushforward operator}
Our main tool for proving Theorems \ref{rank bound} and \ref{degree bound} is to study the coderivative of the 
pullback map $\sigma_f:\mathcal{T}\left(S^2,B\right)\rightarrow\mathcal{T}\left(S^2,A\right)$ induced by an admissible cover $f:\left(S^2,A\right)\rightarrow\left(S^2,B\right)$. 
Suppose 
\begin{equation}\label{points}
\sigma_f\left(\tau=\left[\phi:\left(S^2,B\right)\rightarrow \left(\hat{\Bbb{C}},\phi(B)\right)\right]\right)
=\left[\phi':\left(S^2,A\right)\rightarrow \left(\hat{\Bbb{C}},\phi'(A)\right)\right]
\end{equation}
where $\phi$ and $\phi'$ fit in the commutative diagram \eqref{diagram} for an appropriate rational map $g$.
Recall that the cotangent spaces to $\mathcal{T}\left(S^2,A\right)$ and $\mathcal{T}\left(S^2,B\right)$ at $\tau'=[\phi']$ and $\tau=[\phi]$ can be respectively identified with the spaces $Q^1\left(\hat{\Bbb{C}}-\phi'(A)\right)$ and $Q^1\left(\hat{\Bbb{C}}-\phi(B)\right)$ of integrable quadratic differentials.
In that case, the coderivative
${\rm{D}}^*\sigma_f(\tau)$
is the pushforward map $g_*$ acting on quadratic differentials:
\begin{equation}\label{pushforward}
\begin{cases}
g_*:Q^1\left(\hat{\Bbb{C}}-\phi'(A)\right)\rightarrow Q^1\left(\hat{\Bbb{C}}-\phi(B)\right)\\
q(z)\,{\rm{d}}z^2\mapsto \left[\mathlarger{\sum}_{w\in g^{-1}(z)}q(w)\left(\frac{1}{g'(w)}\right)^2\right]\,{\rm{d}}z^2.
\end{cases}
\end{equation}
See \cite[\S3]{MR1251582} for more details on the derivative of $\sigma_f$, and consult 
\cite{MR2245223,MR3675959} for the necessary background from Teichmüller theory.

\subsection{Proof of Theorem \ref{rank bound}}
\begin{proof}[Proof of Theorem \ref{rank bound}]
With the notation as above, assuming that $\sigma_f$ takes $\tau=[\phi]$ to $\tau'=[\phi']$ as in \eqref{points} where $f$ and the
representatives $\phi$ and $\phi'$ make the diagram \eqref{diagram} commutative for an appropriate rational function $g$, we identify 
${\rm{D}}^*\sigma_f(\tau):{\rm{T}}^*_{\tau'}\mathcal{T}\left(S^2,A\right)\rightarrow{\rm{T}}^*_{\tau}\mathcal{T}\left(S^2,B\right)$
with the pushforward map $g_*:Q^1\left(\hat{\Bbb{C}}-\phi'(A)\right)\rightarrow Q^1\left(\hat{\Bbb{C}}-\phi(B)\right)$ from \eqref{pushforward}.
The spaces $Q^1\left(\hat{\Bbb{C}}-\phi'(A)\right)$ and $Q^1\left(\hat{\Bbb{C}}-\phi(B)\right)$ of integrable quadratic differentials are of dimensions $|A|-3$ and $|B|-3$ respectively.  Without any loss of generality, we may pick the representatives $\phi$ and $\phi'$ of the Teichm\"uller points $\tau$ and $\tau'$ such that 
$\phi'(A),\phi(B)\subset\Bbb{C}$. Denote $\dim\, \ker{\rm{D}}^*\sigma_f(\tau)$ by $k$. If $k=0$, then 
$$
{\rm{rank}}\, {\rm{D}}^*\sigma_f(\tau)={\rm{rank}}\, {\rm{D}}\sigma_f(\tau)=\dim\, \mathcal{T}\left(S^2,A\right)=|A|-3,
$$
and inequality \eqref{inequality1} clearly holds.  So suppose 
$\{q_1(z)\,{\rm{d}}z^2,\dots,q_k(z)\,{\rm{d}}z^2\}$  
is a basis for $\ker g_*$ where $q_1(z),\dots,q_k(z)$ are linearly independent rational functions on $\hat{\Bbb{C}}$.
Denote the set of poles of the quadratic differentials $q_j(z)\,{\rm{d}}z^2\, (1\leq j\leq k)$ by 
$X:=\{u_1,\dots,u_n\}\subseteq\phi'(A)$ where the $u_i$'s are distinct complex numbers, and none of them is a pole of $g$.
Notice that we must have $n\geq 4$ because any integrable quadratic differential has at least four poles. \\
\indent
Each $q_j(z)\,{\rm{d}}z^2$ can be written as 
$$
\frac{p_j(z)}{\prod_{i=1}^n (z-u_i)}\,{\rm{d}}z^2
$$
where $p_j(z)$ is a polynomial with $\deg p_j\leq n-4$  (since otherwise there would be a pole at infinity). 
Let $c_*$ be a simple critical point of $g$  such that all other points of 
$g^{-1}(g(c_*))$ are regular, and this fiber does not intersect $\phi'(A)$.  
We shall show that any quadratic differentials $q_j(z)\,{\rm{d}}z^2$ vanishes at $c_*$. 
Let us first see what implication this would have for the rank.
The set of critical points $c_*$ with the property mentioned above is the image under $\phi'$ of the subset from \eqref{ell2}, i.e. 
$$
\left\{c\in S^2-A \,\Bigg|\,
\begin{array}{l}
c \text{ is  a } \text{simple critical point of  } f \text{ with } f^{-1}(f(c))\cap A=\emptyset\\
\text{and no  other point of } f^{-1}(f(c))\text{ is critical}. 
\end{array}
\right\}
$$
where $c_*=\phi'(c)$. The size of this set is denoted by $\ell_2$ in \eqref{ell2}. Therefore, the aforementioned claim requires polynomials $p_1(z),\dots,p_k(z)$ to vanish on a set of size $\ell_2$. The dimension of the space of polynomials of degree at most $n-4$ that vanish on a set of size $\ell_2$ is $n-3-\ell_2$ (if $c_*=\infty$, then $\deg p_j\leq n-4$ must be replaced with $\deg p_j\leq n-5$). 
We deduce that 
\begin{equation}\label{auxiliary3}
{\rm{rank}}\, {\rm{D}}\sigma_f(\tau)={\rm{rank}}\, {\rm{D}}^*\sigma_f(\tau)=|A|-3-k\geq (|A|-n)+\ell_2.
\end{equation}
The other part of the proof is to provide an upper bound for $n$ by showing that 
\begin{equation}\label{auxiliary4}
n\leq \left|\left\{a\in A\mid a \text{ is in a critical fiber or }|A\cap f^{-1}(f(a))|>1\right\}\right|.
\end{equation}
Substituting in \eqref{auxiliary3}, this would yield the desired lower bound \eqref{inequality1} for the rank:
\begin{equation*}
\begin{split}
{\rm{rank}}\, {\rm{D}}\sigma_f(\tau)
&\geq\left(|A|-\left|\left\{a\in A\mid a \text{ is in a critical fiber or }|A\cap f^{-1}(f(a))|>1\right\}\right|\right)+\ell_2\\
&=\left|\left\{a\in A\mid a \text{ is not in a critical fiber and }|A\cap f^{-1}(f(a))|=1\right\}\right|+\ell_2\\
&=\left|\{b\in B\mid b \text{ is not a critical value of } f \text{ and } |A\cap f^{-1}(b)|=1\}\right|+\ell_2\\
&=\ell_1+\ell_2.
\end{split}
\end{equation*}
To conclude the proof, we first establish \eqref{auxiliary4} and then the claim made on $c_*$. \\
\indent
Inequality \eqref{auxiliary4} will be obtained by investigating  $X=\{u_1,\dots,u_n\}$. Each element $u_i$ of this set should occur as a pole of one of quadratic differentials  
\begin{equation}\label{differentials}
q_1(z)\,{\rm{d}}z^2=\frac{p_1(z)}{\prod_{i=1}^n (z-u_i)}\,{\rm{d}}z^2,\dots,
q_k(z)\,{\rm{d}}z^2=\frac{p_k(z)}{\prod_{i=1}^n (z-u_i)}\,{\rm{d}}z^2.
\end{equation}
Thus one of polynomials $p_1(z),\dots,p_k(z)$ must be nonzero at $u_i$, i.e. these polynomials do not simultaneously vanish at any point from $X=\{u_1,\dots,u_n\}$. Hence one can pick a polynomial  $p(z)$ belonging to 
${\rm{Span}}\{p_1(z),\dots,p_k(z)\}$ which is nonzero at every element of $X$. Now the quadratic differential
\begin{equation}\label{differential}
q(z)\,{\rm{d}}z^2:=\frac{p(z)}{\prod_{i=1}^n (z-u_i)}\,{\rm{d}}z^2
\end{equation}
has $X$ as the set of its (simple) poles, and also lies in $\ker g_*$. 
So we should have 
\begin{equation}\label{pushzero}
\mathlarger{\sum}_{w\in g^{-1}(z)}\frac{p(w)}{\prod_{i=1}^n (w-u_i)}\left(\frac{1}{g'(w)}\right)^2=0\quad \forall z.
\end{equation}
The key point is to notice that  since the sum is identically zero, if a summand tends to infinity, some other summand must blow up as well.
Fix an 
$i_*\in\{1,\dots,n\}$. As $z\to g(u_{i_*})$, one has $w=g^{-1}(z)\to u_{i_*}$ for an appropriate branch of $g^{-1}$; the corresponding summand $\frac{p(w)}{\prod_{i=1}^n (w-u_i)}\left(\frac{1}{g'(w)}\right)^2$ thus tends to infinity. (Keep in mind that $p(u_{i_*})\neq 0$, and $u_{i_*}\in\Bbb{C}$ is not a pole of $g$ or $g'$.) Thus there should also be another branch $w=g^{-1}(z)$ for which $\frac{p(w)}{\prod_{i=1}^n (w-u_i)}\left(\frac{1}{g'(w)}\right)^2\to\infty$ as $z\to g(u_{i_*})$. If $w=g^{-1}(z)\to u_{i_*}$ for this branch too, then $u_{i_*}$ must be a critical point of $g$, and the term $g'(w)$ in the denominator tends to $0$. Now suppose $g(u_{i_*})$ is not a critical value of $g$. Then, as 
$z\to g(u_{i_*})$, this other branch $w=g^{-1}(z)$ should tend to a point of the fiber $g^{-1}(g(u_{i_*}))$ different from $u_{i_*}$, and $g'$ is nonzero at that limit. Consequently, the limit of  $w=g^{-1}(z)$ should be another  point 
of $X=\{u_1,\dots,u_n\}$ different from $u_{i_*}$ since otherwise  
$\frac{p(w)}{\prod_{i=1}^n (w-u_i)}\left(\frac{1}{g'(w)}\right)^2$
cannot blow up as $z\to g(u_{i_*})$. But then the value of $g$ at that point coincides with $g(u_{i_*})$.
We conclude that: each point of $X$ either lies above a critical value of $g$, or the fiber of $g$ it belongs to contains more than one element of $X$. 
Pulling back via homeomorphisms $\phi$ and $\phi'$, we obtain the subset 
$X':=\phi'^{-1}(X)\subseteq A$ of the same size $n$ which satisfies a similar property: for any $a\in X'$, either $f(a)\in B$ is a critical value of $f:\left(S^2,A\right)\rightarrow\left(S^2,B\right)$, or $A\cap f^{-1}(f(a))$ has more than one element. Inequality \eqref{auxiliary4} then follows. \\
\indent
We finally turn to the claim made about the critical point $c_*$: if ${\rm{mult}}_{c_*}(g)=2$,  and the multiplicity of any other point above $v_*:=g(c_*)$ is $1$, and no point above $v_*$ lies in $A$,
then all elements of the basis \eqref{differentials} for $\ker g_*$ must vanish at $c_*$. 
There is no loss of generality in assuming that $v_*\neq\infty$ and $\infty\notin g^{-1}(v_*)$.
Aiming for a contradiction, suppose one of the polynomials $p_j$ is nonzero at $c_*\in\Bbb{C}$. Repeating the argument used before, one can pick a polynomial
$p(z)\in{\rm{Span}}\{p_1(z),\dots,p_k(z)\}$ that is nonzero not only for elements of $X=\{u_1,\dots,u_n\}$, but also at $c_*$. 
We shall again work with the quadratic differential 
$\frac{p(z)}{\prod_{i=1}^n (z-u_i)}\,{\rm{d}}z^2$ from \eqref{differential} that
belongs to $\ker g_*$, so \eqref{pushzero} is satisfied. 
We next rewrite the sum \eqref{pushzero}. 
Let $V$ be a small open subset around $v_*\in\Bbb{C}$ whose preimage $g^{-1}(V)$ may be written as a union of disjoint open sets  $U_1,\dots,U_r$ such that each $U_j$ contains exactly one element of $g^{-1}(v_*)$. 
Let us denote the unique element of $U_j\cap g^{-1}(v_*)$ by $c_{*j}\in\Bbb{C}$. By shrinking $V$ if necessary, in a suitable coordinate system, $g\restriction_{U_j}:U_j\rightarrow V$ is in the form of $\tilde{w}\mapsto \tilde{w}^{m_j}$ where   
$m_j:={\rm{mult}}_{c_{*j}}(g)$.
We then have:
\begin{equation}\label{sum} 
0=\mathlarger{\sum}_{w\in g^{-1}(z)}\frac{p(w)}{\prod_{i=1}^n (w-u_i)}\left(\frac{1}{g'(w)}\right)^2
=\mathlarger{\sum}_{j=1}^r
\mathlarger{\sum}_{w\in g^{-1}(z)\cap U_j}\frac{p(w)}{\prod_{i=1}^n (w-u_i)}\left(\frac{1}{g'(w)}\right)^2
\quad\forall z\in V.
\end{equation}
As $z\to v_*\in\Bbb{C}$, in \eqref{sum}, the points $w$ lying in 
$g^{-1}(z)\cap U_j$ tend to $c_{*j}\in\Bbb{C}$ along $m_j$ different branches of $g^{-1}$. 
The corresponding sum 
$\mathlarger{\sum}_{w\in g^{-1}(z)\cap U_j}\frac{p(w)}{\prod_{i=1}^n (w-u_i)}\left(\frac{1}{g'(w)}\right)^2$
can blow up only if $c_{*j}$ is a critical point or is one of the $u_i$'s. 
Terms $p(w)$ in numerators remain bounded; so one should discuss the growth rate of 
$\mathlarger{\sum}_{w\in g^{-1}(z)\cap U_j}\frac{1}{\prod_{i=1}^n (w-u_i)}\left(\frac{1}{g'(w)}\right)^2$
as $z$ approaches $v_*$.
Lemma \ref{technical} below investigates the growth rate of this fraction 
as $z$ approaches $v_*$ along a line segment, i.e. $z=t+v_*\to v_*$ where $t\to 0^+$.  
Given that $p(c_*)\neq 0$, the lemma implies  that 
when $c_{*j}$ is the critical point $c_*\in \Bbb{C}-\{u_1,\dots,u_n\}$, the sum   
$\mathlarger{\sum}_{w\in g^{-1}(z)\cap U_j}\frac{p(w)}{\prod_{i=1}^n (w-u_i)}\left(\frac{1}{g'(w)}\right)^2$ 
is $\frac{C}{t}+O(1)$ as $t\to 0^+$ for some nonzero constant $C$. 
All other points $c_{*j}$ are regular and outside $\phi'(A)\supseteq X=\{u_1,\dots,u_n\}$; 
therefore,  their corresponding sums are $O(1)$. 
In summary, as  $z=t+v_*\to v_*$, among the $r$ summations appearing on the right-hand side of \eqref{sum}, there is a single one which is $\frac{C}{t}+O(1)$ while the rest are bounded. This is absurd because they should add up to zero.
\end{proof}

Here is the lemma used in the proof above:
\begin{lemma}\label{technical}
Let $g:\hat{\Bbb{C}}\rightarrow\hat{\Bbb{C}}$ be a rational map. Suppose $v_*$ is a complex number,  and $c_*\in\hat{\Bbb{C}}$ is a simple critical point belonging to $f^{-1}(v_*)$.  Take $U$ to be an open subset that intersects $g^{-1}(v_*)$ only at $c_*$. Assume that we are tending to $v_*$ along the line segment 
\begin{equation}\label{path}
\mathbf{r}(t):=t+v_*; 
\end{equation}
and take $X$ to be $\{u_1,\dots,u_n\}$ where $u_1,\dots,u_n$ are pairwise distinct complex numbers different from $c_*$.  
Then, as $t\to 0^+$, 
\begin{equation}\label{local sum}
\mathlarger{\sum}_{w\in U,\, g(w)=\mathbf{r}(t)}\frac{1}{\prod_{i=1}^n (w-u_i)}\left(\frac{1}{g'(w)}\right)^2 
=\frac{C}{t}+O(1)
\end{equation}
where $C:=\frac{1}{\prod_{i=1}^n (c_*-u_i)}\frac{1}{g''(c_*)}$
is a nonzero constant.
\end{lemma}

\begin{proof}
The local behavior of $g$ near $c_*$ can be described as $\tilde{w}\mapsto \tilde{w}^2$ in appropriate coordinate systems. 
In other words, by shrinking $U$ if necessary, one can find a complex chart $\eta$ on $U$ with $\eta(c_*)=0$ so that  
\begin{equation}\label{normal form}
g(w)=\left(\eta(w)\right)^2+v_*\quad \forall w\in U.
\end{equation} 
The curve $\mathbf{r}$ from \eqref{path} lifts to two curves in $U$ passing through $g^{-1}(v_*)\cap U=\{c_*\}$. They can be described as 
\begin{equation}\label{paths}
t\mapsto\eta^{-1}\left(\pm\sqrt{t}\right)\quad t\geq 0.
\end{equation} 
Substituting in \eqref{local sum}, one needs to analyze 
\begin{equation}\label{local sum'}
\begin{split}
&\mathlarger{\sum}_{j=0}^{1}\,\frac{1}
{\prod_{i=1}^n \left(\eta^{-1}\left((-1)^j\sqrt{t}\right)-u_i\right)}
\left(\frac{1}{2\,\eta'\left(\eta^{-1}\left((-1)^j\sqrt{t}\right)\right)
\eta\left(\eta^{-1}\left((-1)^j\sqrt{t}\right)\right)}\right)^2\\
&=\left[\frac{1}{\prod_{i=1}^n \left(\eta^{-1}(\sqrt{t})-u_i\right)}\frac{1}{4\,\left(\eta'\left(\eta^{-1}\left(\sqrt{t}\right)\right)\right)^2}
+\frac{1}{\prod_{i=1}^n \left(\eta^{-1}(-\sqrt{t})-u_i\right)}\frac{1}{4\,\left(\eta'\left(\eta^{-1}\left(-\sqrt{t}\right)\right)\right)^2}\right]\frac{1}{t}
\end{split}
\end{equation}
as $t\to 0^+$. 
The term in brackets is continuous at $t=0$ (the point at which $\eta^{-1}=c_*$) and tends to  
\begin{equation}\label{C}
C:=\frac{1}{\prod_{i=1}^n (c_*-u_i)}\frac{1}{2\left(\eta'(c_*)\right)^2}\neq 0
\end{equation}
which is finite and nonzero (recall that $c_*\notin\{u_1,\dots,u_n\}$).
We conclude that \eqref{local sum} goes to infinity like 
$$\left(C+O(t)\right)\frac{1}{t}=\frac{C}{t}+O(1).$$ 
The constant $C$ above is the desired one because differentiating \eqref{normal form} along with $\eta(c_*)=0$ yield 
$2\left(\eta'(c_*)\right)^2=g''(c_*)\neq 0$.
\end{proof}

\begin{remark}
Possible poles of the pushforward of an integrable quadratic differential $q(z)\,{\rm{d}}z^2$ by a rational map $g$  are critical values of $g$ or the images of poles of $q(z)\,{\rm{d}}z^2$ under $g$. In Lemma \ref{technical}, the assumption on the nondegeneracy of the critical point  is to ensure that the pushforward admits a pole. There does not seem to be any easy formulation of the lemma for degenerate critical points. Indeed, in local coordinates and in terms of Laurent series, the pushforward of 
$\left(\sum_{k=-1}^\infty a_k\tilde{w}^k\right)\,{\rm{d}}\tilde{w}^2$
by $\tilde{w}\mapsto\tilde{w}^m$ is given by
$\left(\sum_{j=-1}^\infty\frac{a_{m(j+2)-2}}{m}\tilde{w}^{j}\right)\,{\rm{d}}\tilde{w}^2$. 
For this pushforward to have a pole, the coefficient $a_{m-2}$ should be nonzero. 
This has a global interpretation in terms of the original quadratic differential only when $m=2$: in that case, one can just require the differential to be holomorphic and nonzero at the point under consideration (i.e. $a_{-1}=0$ and $a_{0}\neq 0$).
\end{remark}

\subsection{Proof of Theorem \ref{degree bound}}

\begin{proof}[Proof of Theorem \ref{degree bound}]
We first obtain the  inequality $|A|\leq d+2$. Similar to the beginning of the proof of Theorem \ref{rank bound}, we have 
$$\sigma_f\left(\tau=\left[\phi:\left(S^2,B\right)\rightarrow \left(\hat{\Bbb{C}},\phi(B)\right)\right]\right)
=\left[\phi':\left(S^2,A\right)\rightarrow \left(\hat{\Bbb{C}},\phi'(A)\right)\right]
$$
which yields the commutative diagram \eqref{diagram}.
The coderivative ${\rm{D}}^*\sigma_f(\tau)$, identified  with the pushforward map $g_*:Q^1\left(\hat{\Bbb{C}}-\phi'(A)\right)\rightarrow Q^1\left(\hat{\Bbb{C}}-\phi(B)\right)$ as before, is identically zero. 
The key idea is to focus on the action of $g_*$ on those integrable quadratic differentials which have exactly four poles.
Without any loss of generality, we may assume $\infty\in\phi'(A)$. Write the points in $\phi'(A)$ as 
$\phi'(A)=\{u_1,\dots,u_n,\infty\}$ where  the $u_i$'s are distinct points in $\Bbb{C}$ and $n:=|A|-1\geq 3$ (nothing to prove when $|A|=3$).
A quadratic differential with at worst simple poles at the points of $\phi'(A)$ is of the form 
$\frac{p(z)}{\prod_{i=1}^n (z-u_i)}\,{\rm{d}}z^2$
where $p$ is a polynomial of degree not greater than $n-3$ (a larger degree causes a pole of higher order at infinity).
One can take $p$ to be a polynomial which has $n-3$ of the points 
$u_1,\dots,u_n$ as its roots to obtain quadratic differentials of the form 
$\frac{1}{(z-u_r)(z-u_s)(z-u_t)}\,{\rm{d}}z^2$
where $1\leq r<s<t\leq n$.
Notice that such quadratic differentials span the cotangent space:
$$Q^1\left(\hat{\Bbb{C}}-\phi'(A)\right)={\rm{Span}}\left\{\frac{\prod_{i\in I}(z-u_i)}{\prod_{i=1}^n (z-u_i)}\,{\rm{d}}z^2\right\}_{I\subset \{1,\dots, n\}, |I|=n-3},$$
because $\left\{\prod_{i\in I}(z-u_i)\right\}_{I\subset \{1,\dots, n\}, |I|=n-3}$
spans the vector space of polynomials of degree not greater than $n-3$. 
We conclude that ${\rm{D}}^*\sigma_f(\tau)=0$ may be rephrased as 
$g_*\left(\frac{1}{(z-u_r)(z-u_s)(z-u_t)}\,{\rm{d}}z^2\right)=0$ for any $1\leq r<s<t\leq n$. 
By the definition of $g_*$ from \eqref{pushforward}: 
$$
\mathlarger{\sum}_{w\in g^{-1}(z)}\frac{1}{(w-u_r)(w-u_s)(w-u_t)}\left(\frac{1}{g'(w)}\right)^2=0\quad \forall z.
$$
Next,  set $r=1$, $s=2$ and $t>2$ in the equation above. Let $z_*\in\Bbb{C}-\{g(\infty),g(u_1),\dots,g(u_t)\}$ be a regular value for the rational function $g$ and write the elements of the fiber $g^{-1}(z_*)$ as $\{w_1,\dots,w_d\}$;   the $w_j$'s are distinct complex numbers not belonging to $\{u_1,\dots,u_n\}$. We then have:
\begin{equation}\label{auxiliary6}
\mathlarger{\sum}_{j=1}^d\frac{1}{(w_j-u_1)(w_j-u_2)(w_j-u_t)}\left(\frac{1}{g'(w_j)}\right)^2=0 \quad 
\forall t\in\{3,\dots,n\}\quad \left(\{w_1,\dots,w_d\}=g^{-1}(z_*)\right).
\end{equation}
This can be thought of as a homogeneous linear system of equations with $d$ unknowns 
$$\left(\frac{1}{g'(w_1)}\right)^2,\dots,\left(\frac{1}{g'(w_d)}\right)^2$$ 
and $n-2$ equations indexed by $t\in\{3,\dots,n\}$. But here the complex numbers $\left(\frac{1}{g'(w_j)}\right)^2\neq 0$ provide a nontrivial solution (keep in mind that $w_j$ is neither a pole nor a critical point). This is possible only if the rank of  coefficient matrix 
$$
\left[
\frac{1}{(w_j-u_1)(w_j-u_2)(w_j-u_t)}
\right]_{\substack{3\leq t\leq n\\ 1\leq j\leq d}}
$$
is smaller than the number of unknowns $d$. But the matrix above -- which is of size $(n-2)\times d$ -- is full rank; its rank is $m:=\min\{n-2,d\}$. To see this, we just need to show that the minor 
$$
\left[\frac{1}{(w_j-u_1)(w_j-u_2)(w_j-u_t)}\right]_{\substack{3\leq t\leq m+2\\ 1\leq j\leq m}}
$$
has nonzero determinant. 
By taking out the term
$\frac{1}{(w_j-u_1)(w_j-u_2)}$
from every entry of the $j^{\rm{th}}$ column for all $1\leq j\leq m$, and then using the classical formula for \textit{Cauchy determinants}, we obtain the determinant as 
\begin{equation*}
\begin{split}
&\frac{1}{\prod_{j=1}^m(w_j-u_1)\prod_{j=1}^m(w_j-u_2)}.\det\left(\left[\frac{1}{(w_j-u_t)}\right]_{\substack{3\leq t\leq m+2\\ 1\leq j\leq m}}\right)\\
&=\frac{1}{\prod_{j=1}^m(w_j-u_1)\prod_{j=1}^m(w_j-u_2)}.\frac{\prod_{t=4}^{m+2}\prod_{j=1}^{t-3}
(w_{t-2}-w_j)(u_{j+2}-u_t)}{\prod_{t=3}^{m+2}\prod_{j=1}^m(w_j-u_t)}.
\end{split}
\end{equation*}
This is nonzero since $u_1,u_2,u_3,\dots,u_{m+2},\dots,u_n,w_1,\dots,w_m,\dots,w_d$ are pairwise distinct complex numbers.
Consequently, the rank $m$ is equal to $\min\{n-2=|A|-3,d\}$. As mentioned before, this must be smaller than $d$ which concludes the proof of $|A|\leq d+2$.  \\
\indent
We next turn to the second part: $v$ is a critical value of $f$ whose fiber is disjoint from $A$ and does not contain any degenerate critical points. By changing the representative $\phi$ of $\tau$ if necessary, we may assume that $v_*:=\phi(v)\neq \infty$. By the commutative diagram \eqref{diagram}, $v_*\in\Bbb{C}$ is a critical value of $g$, the fiber $g^{-1}(v_*)$ does not intersect $\phi'(A)=\{u_1,\dots,u_n,\infty\}$, and the critical points belonging to $g^{-1}(v_*)$ are all simple.  
Equation \eqref{auxiliary6} is satisfied for  regular values $z_*$ near $v_*$.
Suppose $z_*=t+v_*$ where $t\to 0^+$.  As $z_*$ tends to $v_*$, elements $w_j$ of $g^{-1}(z_*)$ tend to points in  
$$g^{-1}(v_*)\subset\hat{\Bbb{C}}-\left(g^{-1}(\infty)\cup\{u_1,\dots,u_n,\infty\}\right).$$
We now invoke Lemma \ref{technical}. 
As $t\to 0^+$, the limit of a $w_j\in g^{-1}(z_*)$ is either a regular point in which case the corresponding summand in \eqref{auxiliary6} is $O(1)$; or that limit is a simple critical point $c_*$ to which two different preimages $w_j=g^{-1}(z_*)$ converge, and the corresponding summands from \eqref{auxiliary6} add up to 
$$
\left(\frac{1}{(c_*-u_1)(c_*-u_2)(c_*-u_t)}\frac{1}{g''(c_*)}\right)\frac{1}{t}+O(1).
$$
The sum of coefficients of $\frac{1}{t}$ in Laurent series must be zero:
\begin{equation}\label{auxiliary7}
\sum_{c_*\in g^{-1}(v_*)\text{ a critical point}}\frac{1}{(c_*-u_1)(c_*-u_2)(c_*-u_t)}\frac{1}{g''(c_*)}=0
\quad\forall t\in\{3,\dots,n\}.
\end{equation}
The logic applied to \eqref{auxiliary6} may be repeated: \eqref{auxiliary7} defines a homogeneous linear system with 
$n-2=|A|-3$ equations and 
$$
\left|g^{-1}(v_*)\cap{\rm{C}}_g\right|=\left|f^{-1}(v)\cap{\rm{C}}_f\right|
$$
unknowns. The unknowns $\frac{1}{g''(c_*)}$ are nonzero (keep in mind that $c_*$ is a simple critical point, and not a pole); and the matrix of coefficients is full rank.  Consequently, the number of unknowns must be greater than the number of equations:
$$
\left|f^{-1}(v)\cap{\rm{C}}_f\right|>|A|-3.
$$
This is the desired improved inequality from Theorem \ref{degree bound}.
\end{proof}

\begin{remark}
If $|B|=3$, then for any admissible branched cover $f:\left(S^2,A\right)\rightarrow\left(S^2,B\right)$ of degree $d$ 
the associated $\sigma_f$ is constant because $\mathcal{T}\left(S^2,B\right)$ is zero-dimensional. 
So one should always have $|A|\leq d+2$ according to Theorem \ref{degree bound}. 
This is indeed true since $A\subseteq f^{-1}(B)$, and 
$|f^{-1}(B)|=d+2$ due to the Riemann-Hurwitz formula.
\end{remark}

\subsection{Corollaries of Theorems \ref{rank bound} and \ref{degree bound}}
We begin with some immediate corollaries of Theorem \ref{rank bound}. The first one was mentioned in \S1. 
\begin{proof}[Proof of Corollary \ref{ell1 corollary}]
Assume the contrary. If there exists an element $a$ of $A$ belonging to a regular fiber that intersects  $A$ only once, then $\ell_1$ from \eqref{ell1} is positive and therefore, Theorem \ref{rank bound} yields 
${\rm{rank}}\, {\rm{D}}\sigma_f\geq\min\{\ell_1,|A|-3\}\geq 1$
at all points of $\mathcal{T}\left(S^2,B\right)$; a contradiction. 
\end{proof}

The two corollaries below follow instantly from Theorem \ref{rank bound}:
\begin{corollary}\label{immersion}
If $f$ has exactly three critical values, $f\restriction_A:A\rightarrow B$ is surjective, and $A$ does not intersect any regular fiber of $f$ at more than one point, then $\sigma_f$ is an immersion.
\end{corollary}

\begin{corollary}\label{submersion}
If  $f\restriction_A:A\rightarrow B$ is injective, and $A$ contains at most three points from the critical fibers of $f$, then $\sigma_f$ is a submersion.
\end{corollary}

\begin{proof}[Proofs of Corollaries \ref{immersion} and \ref{submersion}]
They both follow from ${\rm{rank}}\, {\rm{D}}\sigma_f\geq\min\{\ell_1,|A|-3\}$ with $\ell_1$ as in Theorem \ref{rank bound} (see \eqref{ell1}). In the case of the former corollary, $\ell_1=|B|-3$ and $|A|\geq |B|$; thus,  
${\rm{rank}}\, {\rm{D}}\sigma_f\geq |B|-3=\dim\mathcal{T}\left(S^2,B\right)$ at all points.
In the case of the latter, $\ell_1\geq |A|-3$; thus,  
${\rm{rank}}\, {\rm{D}}\sigma_f\geq |A|-3=\dim\mathcal{T}\left(S^2,A\right)$ at all points.  
\end{proof}

We now turn to the dynamical setting which is concerned with a Thurston map
$f:\left(S^2,{\rm{P}}_f\right)\rightarrow\left(S^2,{\rm{P}}_f\right)$
and its associated pullback 
$\sigma_f:\mathcal{T}\left(S^2,{\rm{P}}_f\right)\rightarrow\mathcal{T}\left(S^2,{\rm{P}}_f\right)$. 
\begin{corollary}\label{periodic bound}
Let $f:S^2\rightarrow S^2$ be a Thurston map  whose critical values are all periodic, and suppose $f$ has at least three critical values. Then, at any point $\tau$ of Teichm\"uller space, one has
$$
{\rm{rank}}\,{\rm{D}}\sigma_f(\tau)\geq |{\rm{P}}_f|-|{\rm{V}}_f|.
$$ 
\end{corollary}
\begin{proof}
All critical values of $f$ are periodic if and only if $f\restriction_{{\rm{P}}_f}:{\rm{P}}_f\rightarrow{\rm{P}}_f$ is a bijection. For any $p\in{\rm{P}}_f-{\rm{V}}_f$, the intersection $f^{-1}(p)\cap{\rm{P}}_f$ is of size one. Therefore, $\ell_1$ from \eqref{ell1} is equal to $|{\rm{P}}_f|-|{\rm{V}}_f|$. Theorem \ref{rank bound} now implies that 
$$
{\rm{rank}}\,{\rm{D}}\sigma_f(\tau)\geq\min\{|{\rm{P}}_f|-|{\rm{V}}_f|,|{\rm{P}}_f|-3\}=|{\rm{P}}_f|-|{\rm{V}}_f|.
$$
\end{proof}

\begin{table}[hbt!]
\begin{tabular}{|p{12cm}|}
\hline
\vspace{1mm}\hfil
\xymatrix
{\underset{v}{\bullet}\ar[r] & \cdots\ar[r]  & \bullet\ar@/^1pc/[ll]& \underset{\infty}{\bullet}\ar@(dr,ur) }\vspace{1mm}\\
\hline
\vspace{1mm}\hfil
\xymatrixcolsep{2.5pc}\xymatrix
{\underset{v}{\bullet}\ar[r] & \cdots \ar[r]& \bullet\ar[r]& \cdots\ar[r]  & \bullet\ar@/^1pc/[ll]& \underset{\infty}{\bullet}\ar@(dr,ur) }\vspace{1mm}
\\
\hline
\end{tabular}
\vspace{2mm}
\caption{Possible postcritical dynamics for a unicritical polynomial $f$ (cf. Corollary \ref{unicritical}). The totally ramified fixed point is shown by $\infty$ and the other critical value by $v$. At the top, $v$ is periodic while it is strictly preperiodic at the bottom.}
\label{Tab:unicritical}
\end{table}

The last corollary can be applied to \textit{topological polynomials}; that is, Thurston maps $f:S^2\rightarrow S^2$ that admit a fixed point $p$ with $f^{-1}(p)=p$. 
\begin{corollary}\label{unicritical}
Let $f$ be a unicritical topological polynomial with  $|{\rm{P}}_f|\geq 3$. The pullback map 
$\sigma_f:\mathcal{T}\left(S^2,{\rm{P}}_f\right)\rightarrow\mathcal{T}\left(S^2,{\rm{P}}_f\right)$ 
is a local isomorphism.
\end{corollary}
\begin{proof}
The Thurston map $f$ admits a totally ramified fixed point, which (in analogy with complex polynomials) we denote by $\infty$, and another critical value $v$ which is eventually periodic. The orbit of $v$ along with $\infty$ comprise the postcritical set ${\rm{P}}_f$. As demonstrated in  Table \ref{Tab:unicritical}, there are two possibilities for $v$: it is either periodic or strictly preperiodic. In the former case $f^{\circ n}(v)=v$ whereas in the latter $f^{\circ n}(v)=f^{\circ m}(v)$ where $m,n$ are appropriate positive integers with $m>n$. The postcritical set is 
${\rm{P}}_f=\left\{v,f(v),\dots,f^{\circ (n-1)}(v),\infty\right\}$    
or 
${\rm{P}}_f=\left\{v,f(v),\dots,f^{\circ n}(v),\dots,f^{\circ (m-1)}(v),\infty\right\}$    
respectively. Thus the number $\ell_1$ (see \eqref{ell1}) which counts  regular values in ${\rm{P}}_f$ 
that have precisely one preimage in 
${\rm{P}}_f$ is at least $|{\rm{P}}_f|-3$; postcritical points without this property are $v,\infty$ in the first situation and $v,f^{\circ n}(v),\infty$ in the second (cf. the portraits from Table \ref{Tab:unicritical}). In each of these cases inequality \eqref{inequality1} from Theorem \ref{rank bound} implies that the rank of $\sigma_f$ is equal to 
$\dim\mathcal{T}\left(S^2,{\rm{P}}_f\right)=|{\rm{P}}_f|-3$ at every Teichm\"uller point. 
\end{proof}
\noindent
The corollary above can also be obtained from the stronger result \cite[Corollary 5.3]{MR3107522} which indicates that for a topological polynomial $f$ the pullback $\sigma_f$ induces a map in the opposite direction on the level of moduli spaces. 

We next focus on Thurston maps $f:\left(S^2,{\rm{P}}_f\right)\rightarrow\left(S^2,{\rm{P}}_f\right)$
whose associated pullback maps 
$\sigma_f:\mathcal{T}\left(S^2,{\rm{P}}_f\right)\rightarrow\mathcal{T}\left(S^2,{\rm{P}}_f\right)$ 
are constant or, more generally, of rank zero at a point.  The first corollary we prove follows instantly from Theorem \ref{rank bound} whereas the second one requires a combinatorial analysis of 
$f\restriction_{{\rm{P}}_f}:{\rm{P}}_f\rightarrow {\rm{P}}_f$
which is carried out in Proposition \ref{combinatorics}.

\begin{proof}[Proof of Corollary \ref{ell2 corollary}]
By the hypotheses, $f^{-1}(v)$ does not intersect ${\rm{P}}_f$ and contains exactly one critical point, which is simple. The number $\ell_2$ associated with the branched cover $f:\left(S^2,{\rm{P}}_f\right)\rightarrow\left(S^2,{\rm{P}}_f\right)$
as in \eqref{ell2} is thus positive. Theorem \ref{rank bound} then implies that 
${\rm{rank}}\, {\rm{D}}\sigma_f\geq\min\{\ell_2,|{\rm{P}}_f|-3\}\geq 1$ globally. 
\end{proof}

\begin{proof}[Proof of Corollary \ref{main corollary}]
The inequality $|{\rm{P}}_f|\leq d+2$ already appears Theorem \ref{degree bound}. If $p\in {\rm{P}}_f$ is not a critical value, then, being a postcritical point, it must be an iterate of a critical value. Thus it can be written as $f(a)$ for a suitable $a\in {\rm{P}}_f$. 
Applying Corollary \ref{ell1 corollary} to $f:\left(S^2,{\rm{P}}_f\right)\rightarrow\left(S^2,{\rm{P}}_f\right)$ then indicates that $p$ can also be realized as $f(a')$ where $a'\in {\rm{P}}_f-\{a\}$. As such, one needs to show  that $|f^{-1}(p)\cap {\rm{P}}_f|\geq 2$ for all
$p\in {\rm{P}}_f-{\rm{V}}_f$  results in $|{\rm{P}}_f|\leq 2|{\rm{V}}_f|$. This follows from  Proposition \ref{combinatorics} below.
\end{proof}
\noindent
The preceding proof relies on the next proposition which illuminates the dynamics on ${\rm{P}}_f$. 
\begin{proposition}\label{combinatorics}
Let $f:S^2\rightarrow S^2$ be a Thurston map. 
Suppose for any $p\in {\rm{P}}_f- {\rm{V}}_f$ one has $|f^{-1}(p)\cap {\rm{P}}_f|\geq 2$. 
Then 
\begin{equation}\label{inequality2'}
|{\rm{P}}_f|\leq 2\,|{\rm{V}}_f|.
\end{equation}
The equality is achieved if and only if for every critical value $v$ of $f$
\begin{itemize}
\item $f^{\circ n}(v)$ is not a critical value of $f$ for any $n\geq 1$;
\item $\left|f^{-1}\left(f^{\circ n}(v)\right)\cap {\rm{P}}_f\right|=2$  for any $n\geq 1$.
\end{itemize} 
\end{proposition}

\begin{figure}[ht!]
\center
\includegraphics[width=12cm,height=6cm]{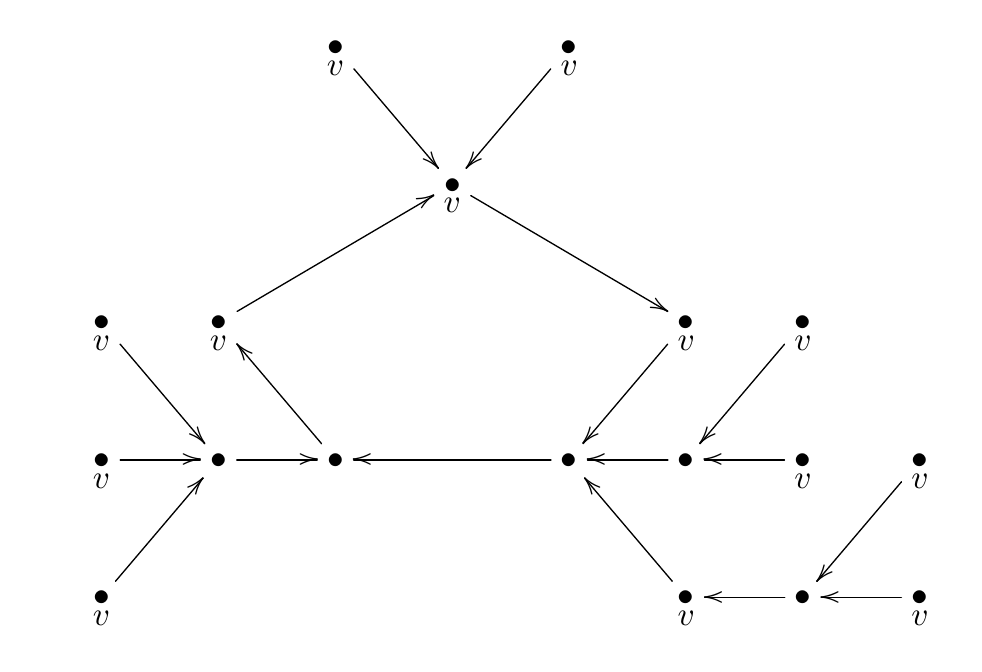}
\caption{An illustration of  a connected component $\mathcal{G}'$ of the functional graph $\mathcal{G}_f$ for the postcritical dynamics  $f\restriction_{{\rm{P}}_f}:{\rm{P}}_f\rightarrow {\rm{P}}_f$.
It is the union of a directed cycle 
 (i.e. a periodic orbit of $f\restriction_{{\rm{P}}_f}:{\rm{P}}_f\rightarrow {\rm{P}}_f$) and possibly some rooted trees    whose roots lie on the cycle (they capture those postcritical points that eventually land on the aforementioned periodic orbit after repeated applications of $f$).
The vertices marked by $v$ correspond to critical values; that is, elements of ${\rm{V}}_f\subseteq{\rm{P}}_f$. 
The assumption here is that ${\rm{D}}\sigma_f$ vanishes somewhere, so the indegree of any vertex not marked by $v$ is at least two  due to Corollary \ref{ell1 corollary}. }
\label{fig:component}
\end{figure}

\begin{proof}
The proof is combinatorial and is based on analyzing the functional graph of 
$f\restriction_{{\rm{P}}_f}:{\rm{P}}_f\rightarrow {\rm{P}}_f$ which we denote by $\mathcal{G}_f$. 
This is a finite directed graph whose vertices are points of ${\rm{P}}_f=\bigcup_{n\geq 0}f^{\circ n}({\rm{V}}_f)$ with an arrow from $p$ to $q$ if and only if $f(p)=q$. We mark those vertices that correspond to critical values of $f$ (i.e. elements of ${\rm{V}}_f$) by $v$. The graph $\mathcal{G}_f$ should satisfy the following properties:
\begin{itemize}
\item The outdegree of each vertex is exactly $1$;
\item vertices with indegree $0$ are marked by $v$;
\item if a vertex is not marked by $v$, then its indegree should be larger than $1$.
\end{itemize}
To prove \eqref{inequality2'}, it suffices to show that in each connected component of $\mathcal{G}_f$, the number of vertices marked by $v$ is at least half the number of vertices in the component. 
Let $\mathcal{G}'$ be a connected component of $\mathcal{G}_f$. It consists of a periodic orbit of  $f\restriction_{{\rm{P}}_f}:{\rm{P}}_f\rightarrow {\rm{P}}_f$ and all the points of ${\rm{P}}_f$ that eventually land on it under iteration; see Figure \ref{fig:component}. \\
\indent
We have
\begin{equation}\label{auxiliary1}
\begin{split}
\#\text{vertices of } \mathcal{G}' \text{ not marked by }v 
&\leq\#\text{vertices of } \mathcal{G}' \text{ with indegree larger than }1\\ 
&\leq\frac{1}{2}\left(\#\text{vertices of }\mathcal{G}'\right)
\end{split}
\end{equation}
where the first inequality follows from the properties mentioned above while the second one is because the sum of all indegrees, which is equal to the sum of all outdegrees, is the same as the number of vertices in hand.
We have thus established  
$$
\#\text{vertices of }\mathcal{G}'\leq 2\left(\#\text{vertices of } \mathcal{G}' \text{ marked by }v \right)
$$
which results in the desired inequality $|{\rm{P}}_f|\leq 2\,|{\rm{V}}_f|$. The equality is achieved if and only if one has equality in \eqref{auxiliary1} for all connected components $\mathcal{G}'$ of $\mathcal{G}_f$. This happens precisely when vertices marked by $v$ are of indegree $0$ and the rest are of indegree $2$; namely, the conditions described in the proposition. 
In terms of the functional graph, notice that this means each component  consists a directed cycle each vertex of which has an incoming edge from a directed full binary tree; and the leaves of those trees are the only vertices  marked by $v$.
\end{proof}

\begin{remark}\label{Parry}
Corollary \ref{main corollary} provides upper bounds for the size of the postcritical set of a Thurston map with constant pullback in terms of both its degree and the number of its critical values. 
Sharper bounds are available from unpublished results of Floyd, Parry and Saenz. The arguments therein are of a geometric flavor as opposed to the algebraic/combinatorial nature of techniques employed here.  
To state those results, suppose $f$ is a Thurston map of degree $d$ whose pullback $\sigma_f$ is constant. 
Then \eqref{inequality2} yields $|P_f|\leq 2|V_f|$. Indeed, it is true that $|P_f|\leq |V_f|+1$ \cite{MRParry1}.
Moreover, $|P_f|=O(d)$, also implied by \eqref{inequality2}, can be replaced with  $|P_f|=O\left(\sqrt{d}\right)$ if $f$ is a topological polynomial \cite{MRParry2}.  
\end{remark}

\begin{remark}\label{polynomial}
In the case of topological polynomials, the inequality can be improved to
\begin{equation}\label{inequality2''}
|{\rm{P}}_f|\leq 2\left(\#\textit{finite critical values of } f\right) -1.
\end{equation}
\end{remark}

\section{Applications to maps with constant pullback}

\subsection{Branched covers with three critical values}
In this subsection, we discuss when for an admissible cover $f:\left(S^2,A\right)\rightarrow\left(S^2,B\right)$ with precisely three critical values the pullback map $\sigma_f$ can be constant. 
Rational maps with  three critical values have been studied extensively because of the celebrated Belyi Theorem  \cite{MR697314,MR1913596}. We shall utilize the following feature of such rational maps: once the three critical values on the Riemann sphere are prescribed, then, up to precomposition with holomorphic automorphisms, there are only finitely many rational maps with those critical values in any given degree. 
Indeed, \textit{Belyi maps} -- holomorphic maps $\Sigma\rightarrow\hat{\Bbb{C}}$ defined on a compact Riemann surface $\Sigma$ with critical values in $\{0,1,\infty\}$  -- can be encoded by discrete objects called \textit{dessins d'enfants}; see \cite{MR2895884,MR2036721} for a background on this rich theory. 

\begin{proposition}\label{Belyi}
Let $f:(S^2,A)\rightarrow (S^2,B)$ be an admissible cover with exactly three critical values satisfying $|A|\geq 4$ and $f(A)\nsubseteq{\rm{V}}_f$, and suppose the associated pullback map $\sigma_f$ is constant. 
Then $|A\cap f^{-1}({\rm{V}}_f)|<3$. Moreover, 
if $|A\cap f^{-1}({\rm{V}}_f)|=2$, the branched cover 
$f:S^2\rightarrow S^2$ admits a nontrivial deck transformation.
Finally, if the branched cover $f$ is a rational map  $(\hat{\Bbb{C}},A)\rightarrow (\hat{\Bbb{C}},B)$,
then for any injection $j:B\hookrightarrow\hat{\Bbb{C}}$, there exists a Möbius  transformation $\nu$ that for any $b\in B$ injects $A\cap f^{-1}(b)$ into the fiber $f^{-1}(j(b))$.  
\end{proposition}

\begin{proof}
Suppose $\sigma_f$ is constant. In terms of its Hurwitz correspondence (as discussed in \S2.1), this means that in Diagram \eqref{correspondence}
the map $\rho_A$ from $\mathcal{W}_f=\omega_f\left(\mathcal{T}\left(S^2,B\right)\right)$ to the moduli space $\mathcal{M}_{0,A}$ is constant. 
Recall that the Hurwitz space $\mathcal{W}_f$ is the space of orbits of marked rational maps equivalent to $f$ under the pre and postcomposition action of Möbius transformations \cite[\S2]{MR3107522}. 
In other words, $\mathcal{W}_f$ is the quotient of the space of triples 
$(g,i,j)$ formed by a rational map $g:\hat{\Bbb{C}}\rightarrow\hat{\Bbb{C}}$ of degree $d:=\deg f$ and injections $i:A\hookrightarrow\hat{\Bbb{C}}$, $j:B\hookrightarrow\hat{\Bbb{C}}$  modulo the action 
\begin{equation}\label{action}
(\alpha,\beta)\cdot(g,i,j):=(\beta\circ g\circ\alpha^{-1},\alpha\circ i,\beta\circ j)
\end{equation}
of ${\rm{PSL}}_2(\Bbb{C})\times{\rm{PSL}}_2(\Bbb{C})$
where $g$ fits in a diagram such as 
\begin{equation}\label{diagram'}
\xymatrixcolsep{4pc}\xymatrix{\left(S^2,A\right)\ar[r]^{\phi'}\ar[d]_f &\left(\hat{\Bbb{C}},\phi'(A)\right)\ar[d]^g\\ 
\left(S^2,B\right)\ar[r]_\phi &\left(\hat{\Bbb{C}},\phi(B)\right).}
\end{equation}
Then the point $\left[\left(g,i:=\phi'\restriction_A,j:=\phi\restriction_B\right)\right]$ of $\mathcal{W}_f$  is the image of the Teichmüller point $[\phi]\in\mathcal{T}\left(S^2,B\right)$ under $\omega_f$, and $\rho_A$ sends $[(g,i,j)]$ to the moduli point $[i:A\hookrightarrow\hat{\Bbb{C}}]\in\mathcal{M}_{0,A}$.
Notice that 
$j\circ f\restriction_A=g\circ i$
due to the commutative diagram. 
Denoting the critical points of $f$ by $b_1,b_2,b_3\in B$, 
one can forgo the postcomposition action by normalizing $j=\phi\restriction_B$ so that 
\begin{equation}\label{normalization}
j(b_1)=0, \quad j(b_2)=1, \quad j(b_3)=\infty.
\end{equation}
In that case, the critical values of $g$ from Diagram \eqref{diagram'} are always $0,1, \infty$.
Up to precomposition with Möbius transformations, there are only finitely many possibilities for a rational map of degree $d$ which is ramified only above $0,1, \infty$ because there are only finitely many choices for its monodromy representation.
Therefore, the map $[(g,i,j)]\mapsto [g]$ that sends an element of $\mathcal{W}_f$ with a  representative $(g,i,j)$ in which $j$ is normalized as \eqref{normalization}  to the class of $g$ modulo precomposition takes values in a finite set. Thus, by arbitrarily picking a triple $(g_0,i_0,j_0)$ coming from a commutative diagram such as  \eqref{diagram'} with $j_0$ normalized,  for any other normalized representative $(g,i,j)$ of a point of $\mathcal{W}_f$, the rational map $g$ can be written as $g_0\circ\alpha$ for a suitable Möbius transformation $\alpha$. 
Notice that $(g=g_0\circ\alpha,i,j)$ and $(g_0,\alpha\circ i,j)$ are in the same orbit of action \eqref{action}.
Consequently, every orbit has a representative of the form $(g_0,i,j)$ with $j$ satisfying  \eqref{normalization}. 
This amounts to identifying $\mathcal{W}_f$ with the quotient of the connected component around $(i_0,j_0)$ of 
\begin{equation}\label{normalized}
\left\{(i,j)\mid i:A\hookrightarrow\hat{\Bbb{C}} \text{ and } j:B\hookrightarrow\hat{\Bbb{C}} \text{ satisfy }
j\circ f\restriction_A=g_0\circ i, \text{ and }  
\left(j(b_1), j(b_2), j(b_3)\right)=(0,1,\infty)
\right\}
\end{equation}
by an action of ${\rm{Deck}}(g_0)$ given by $\nu\cdot(i,j):=(\nu\circ i, j)$. We should now investigate the constancy of 
$\rho_A:\mathcal{W}_f\rightarrow\mathcal{M}_{0,A}$. It means that for different pairs $(i,j)$ appearing in \eqref{normalized} the moduli points $[i]\in\mathcal{M}_{0,A}$ are the same. 
As for the projection to $\mathcal{M}_{0,B}$, it is known that $\rho_B:\mathcal{W}_f\rightarrow\mathcal{M}_{0,B}$ is a finite covering map \cite[Theorem 2.6]{MR3107522}. The surjectivity of $\rho_B$ implies that all possible normalized injections $j:B\hookrightarrow\hat{\Bbb{C}}$ occur in the connected component of \eqref{normalized} which contains $(i_0,j_0)$. 
\\
\indent
We now prove the first assertion. 
Notice that for any $a\in A\cap f^{-1}({\rm{V}}_f)=A\cap f^{-1}(\{b_1,b_2,b_3\})$, 
given any $(i,j)$ from \eqref{normalized}, 
the point $i(a)$ lies in the finite set  $g_0^{-1}(\{0,1,\infty\})$ --
which is independent of $i$.
So in the connected component under consideration, $i$ and $i_0$ coincide on the subset $A\cap f^{-1}({\rm{V}}_f)$.  But then if $|A\cap f^{-1}({\rm{V}}_f)|\geq 3$, they must coincide as injections since their classes in $\mathcal{M}_{0,A}$ are the same. This is impossible because there exists a point $a\in A-f^{-1}({\rm{V}}_f)$, and if $i(a)=i_0(a)$, then $j(f(a))=g_0(i_0(a))$. Hence in \eqref{normalized}, normalized injections $j$ that take $f(a)\in B-\{b_1,b_2,b_3\}$ to a point of $\hat{\Bbb{C}}-\{0,1,\infty,g_0(i_0(a))\}$ cannot occur contradicting the surjectivity of 
$\rho_B:\mathcal{W}_f\rightarrow\mathcal{M}_{0,B}$.  \\
\indent
Next, suppose the cardinality of $A\cap f^{-1}({\rm{V}}_f)$ is two. Denote its elements by $a_1,a_2$. As mentioned above, 
$i(a_1)=i_0(a_1)$ and $i(a_2)=i_0(a_2)$. Due to the hypothesis $f(A)\nsubseteq{\rm{V}}_f$, there is a regular value 
$b_4\in B$ whose fiber intersects $A$. The intersection must have at least two points since otherwise $\sigma_f$ cannot be constant; cf. Corollary \ref{ell1 corollary}. Let $a_3$ and $a_4$ be two elements of $A\cap f^{-1}(b_4)$. Then $i_0(a_3)$ and $i_0(a_4)$ lie in the fiber of $g_0$ above $j_0(b_4)$. By the Implicit Function Theorem, if $U$ is a sufficiently small neighborhood of $j_0(b_4)$ in $\hat{\Bbb{C}}$, $i_0(a_3)$ and $i_0(a_4)$ can be continued analytically to points $p_1(u)\neq p_2(u)$ that belong to 
$g_0^{-1}(u)$ for any $u\in U$. We can now continuously perturb $i_0:A\hookrightarrow\hat{\Bbb{C}}$ and 
$j_0:B\hookrightarrow\hat{\Bbb{C}}$ to two other injections
$i:A\hookrightarrow\hat{\Bbb{C}}$ and $j:B\hookrightarrow\hat{\Bbb{C}}$
to obtain another element of $\mathcal{W}_f$: change $j(b_4)$ to a point $u\in U$ near $j_0(b_4)$, and change 
$i_0(a_3)$ and $i_0(a_4)$ to elements $p_1(u)$ and $p_2(u)$ that belong to $g_0^{-1}(u)$; $i$ and $j$ agree with respectively $i_0$ and $j_0$ elsewhere. The conditions desired in \eqref{normalized} are clearly satisfied and thus $(i,j)$ determines a point of the Hurwitz space $\mathcal{W}_f$. But then, due to the constancy of 
$\rho_A:\mathcal{W}_f\rightarrow\mathcal{M}_{0,A}$,  the injections $i_0$ and $i$ differ by a Möbius transformation. In particular, in terms of cross-ratios:
$$T\left(i_0(a_1),i_0(a_2),i_0(a_3),i_0(a_4)\right)
=T\left(i(a_1),i(a_2),i(a_3),i(a_4)\right)=T\left(i_0(a_1),i_0(a_2),p_1(u),p_2(u)\right).$$
There exists a Möbius transformation $\nu\neq{\rm{id}}$ dependent only on $i_0(a_1),i_0(a_2),i_0(a_3),i_0(a_4)\in\hat{\Bbb{C}}$ for which
$$T\left(i_0(a_1),i_0(a_2),i_0(a_3),i_0(a_4)\right)=T\left(i_0(a_1),i_0(a_2),x,y\right)\Leftrightarrow y=\nu(x).$$
Therefore, for any $u$ from the nonempty open set $U$, the distinct points $p_1(u)$ and $p_2(u)$ of $g_0^{-1}(u)$
are related through $p_2(u)=\nu(p_1(u))$. But $g_0$ is analytic, so we should have $g_0=g_0\circ\nu$, i.e. $\nu$ is a nontrivial deck transformation of $g_0:\hat{\Bbb{C}}\rightarrow\hat{\Bbb{C}}$. This can be pulled back to a nontrivial deck transformation  of $f:S^2\rightarrow S^2$ via the commutative diagram.\\
\indent
Finally, we shall prove the third assertion.
There is no loss of generality in assuming that elements $b_1,b_2,b_3$ of $B\subset\hat{\Bbb{C}}$ are $0,1,\infty$ respectively. 
Considering $f$ to be a holomorphic branched cover 
$(\hat{\Bbb{C}},A)\rightarrow (\hat{\Bbb{C}},B)$, in the commutative diagram \eqref{diagram'}, one can take $\phi$ and $\phi'$ to be identity and the right column to be $f$.  
Now consider the distinguished triple 
$$
\left(g_0:=f,i_0:=A\stackrel{inc}{\hookrightarrow}\hat{\Bbb{C}},j_0:=B\stackrel{inc}{\hookrightarrow}\hat{\Bbb{C}}\right)
$$ 
and the previous description of $\mathcal{W}_f$ as the connected component of \eqref{normalized} around $(i_0,j_0)$.
It suffices to prove the last assertion for injections $j:B\hookrightarrow\hat{\Bbb{C}}$
that are normalized in the sense of \eqref{normalization}. The moduli point $[j]\in\mathcal{M}_{0,B}$ can be lifted 
to an element of $\mathcal{W}_f$ via $\rho_B$.  Representing this element by a pair $(i,j)$ as in \eqref{normalized}, one has 
$[i]=[i_0]$ in $\mathcal{M}_{0,A}$ by  invoking the constancy of $\rho_A$ again. Thus $i:A\hookrightarrow\hat{\Bbb{C}}$
should be the composition of the inclusion map with a Möbius map $\nu$. 
The identity $j\circ f\restriction_A=g_0\circ i$ now amounts to $j\circ f\restriction_A=f\circ\nu\restriction_A$. For any arbitrary $b\in B$, taking the preimage of $j(b)$ by both sides implies that
$A\cap f^{-1}(b)$
is included in $\nu^{-1}\left(f^{-1}(j(b))\right)$; this finishes the proof.
\end{proof}

\begin{example}\label{Arturo}
A  remarkable example of a rational map with constant pullback which is not a composition of maps of lower degree appears in the thesis of  A. Saenz  \cite[Appendix D]{MRArturo}.
That example is 
\begin{equation}\label{Arturo example}
f(z):=-\sqrt[3]{2}\,\frac{z(z^3+2)}{2z^3+1};
\end{equation}
the critical points  are the third roots of unity $1,\omega,\omega^2\,(\omega:={\rm{e}}^{\frac{2\pi{\rm{i}}}{3}})$
at any of which the local degree of $f$ is three; they are respectively mapped to 
$-\sqrt[3]{2},-\sqrt[3]{2}\,\omega,-\sqrt[3]{2}\,\omega^2$
via $f$. All these critical values belong to the fiber above the fixed point $z=0$, hence the map is PCF with 
${\rm{P}}_f=\{-\sqrt[3]{2},-\sqrt[3]{2}\,\omega,-\sqrt[3]{2}\,\omega^2,0\}$.
The proof of constancy of $\sigma_f$ in \cite{MRArturo} is geometric: it is based on a  \textit{core arc} argument
that shows that connected components of the preimage of a simple closed curve in $\hat{\Bbb{C}}-{\rm{P}}_f$  are peripheral (which includes the case of the preimages being nullhomotopic too). This implies that $\sigma_f$ is constant due to \cite[Theorem 5.1]{MR3456156}.
A similar argument will later appear in the proof of Proposition \ref{generalization?}.\\
\indent According to Proposition \ref{Belyi}, the regular fibers of the map \eqref{Arturo example} should be conformally the same. This can be also checked directly: $f$ is obtained from composing 
$$
g(z):=-\frac{z(z^3+2)}{2z^3+1}
$$
with the Möbius transformation $z\mapsto\sqrt[3]{2}\,z$. By putting the critical values $-1,-\omega,-\omega^2$
of $g$ at $0,1,\infty$ via a change of coordinates, one obtains the Belyi map 
\begin{equation}\label{Kamalinezhad}
h(z):=-\frac{z(z-2)^3}{(2z-1)^3}
\end{equation} 
whose dessin is demonstrated in Figure \ref{fig:dessin}.
The goal is to verify that the regular fibers of $h$ are isomorphic as subsets of size four of $\hat{\Bbb{C}}$. 
This Belyi map is studied in \cite{MR3431861} for a process called the \textit{flat refinement} of dessins d'enfants. Algebraically, the process amounts to postcomposing a Belyi map with $h$; the result is Belyi again since $\{0,1,\infty\}$ is preserved by $h$ (i.e. $h$ is \textit{Belyi-extending} in the sense of \cite{MR2266994}). Another important observation is that $h$ is \textit{rigid Lattès}: it is induced by a multiplication map on the elliptic curve  $y^2=x^3+1$.
This is all summarized in the commutative diagram below: 
\begin{equation}\label{conjugacies}
\xymatrixcolsep{9pc}\xymatrix{\left\{y^2=x^3+1\right\}\ar[d]_{(x,y)\mapsto\frac{y+1}{2}}\ar[r]^{[-2]} & \left\{y^2=x^3+1\right\} \ar[d]^{(x,y)\mapsto\frac{y+1}{2}}\\
\hat{\Bbb{C}}\ar[r]^{h:z\mapsto-\frac{z(z-2)^3}{(2z-1)^3}} \ar[d]_{z\mapsto\frac{-\omega z-1}{\omega^2 z+1}}& \hat{\Bbb{C}} \ar[d]^{z\mapsto\frac{-\omega z-1}{\omega^2 z+1}}\\
\hat{\Bbb{C}}  \ar[r]^{g:z\mapsto-\frac{z(z^3+2)}{2z^3+1}}& \hat{\Bbb{C}}.}
\end{equation}    
Here, the multiplication map $[-2]$ sends a point $(x,y)$ of the elliptic curve to 
$$
(x,-y)\oplus(x,-y)=\left(\frac{x(x^3-8)}{4(x^3+1)}, -\frac{y^4+18y^2-27}{8y^3}\right).
$$
The commutative diagram above indicates that a generic fiber of 
$h(z)=-\frac{z(z-2)^3}{(2z-1)^3}$
is in the form of 
$$
\left\{\frac{b'+1}{2}\,\Big|\, (a',b')=(a,b)\oplus\text{a 2-torsion point}\right\}.
$$
The non-identity $2$-torsion points of $y^2=x^3+1$ are in the form of $(-\lambda,0)$ where $\lambda\in\{1,\omega,\omega^2\}$ is a $3^{\rm{rd}}$ root of unity. We have
$$
(a,b)\oplus(-\lambda,0)=\left(\left(\frac{b}{a+\lambda}\right)^2-a+\lambda,-\frac{3\lambda^2b}{(a+\lambda)^2}\right);
$$
so the $y$-coordinates of these four points are 
$$
b,-\frac{3b}{(a+1)^2},-\frac{3\,\omega^2\,b}{(a+\omega)^2},-\frac{3\,\omega\,b}{(a+\omega^2)^2}. 
$$
In computing their cross-ratio, the terms miraculously simplify and the cross-ratio turns out to be independent of $(a,b)$:
$$
T\left(b,-\frac{3b}{(a+1)^2},-\frac{3\,\omega^2\,b}{(a+\omega)^2},-\frac{3\,\omega\,b}{(a+\omega^2)^2}\right)
=\frac{1}{1+\omega}.
$$
The details of this computation are left to the reader. For an analytic explanation of why the regular fibers of rational maps such as \eqref{Arturo example} or \eqref{Kamalinezhad} are the same, see \cite[\S 11.3.2]{MR4472056}.
\end{example}

\begin{figure}[ht!]
\includegraphics[width=9cm]{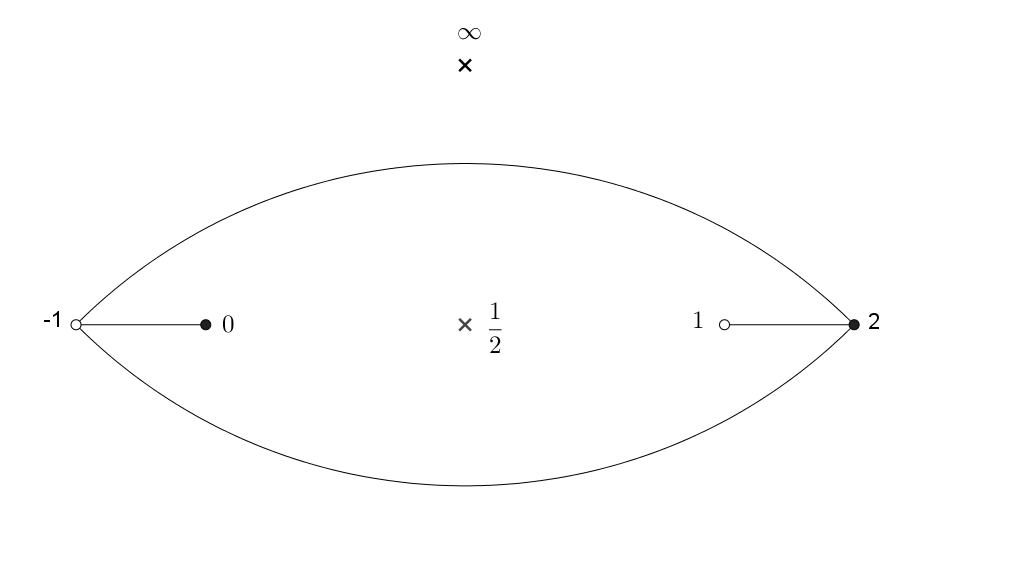}
\caption{The dessin of the Belyi function $h(z)=-\frac{z(z-2)^3}{(2z-1)^3}$ from Example \ref{Arturo example}. 
The vertices of the dessin marked by $\bullet$, $\circ$ and $\times$ are elements of $f^{-1}(0)$, $f^{-1}(1)$ and $f^{-1}(\infty)$ respectively.}
\label{fig:dessin}
\end{figure}

\begin{figure}[ht!]
\includegraphics[width=13cm]{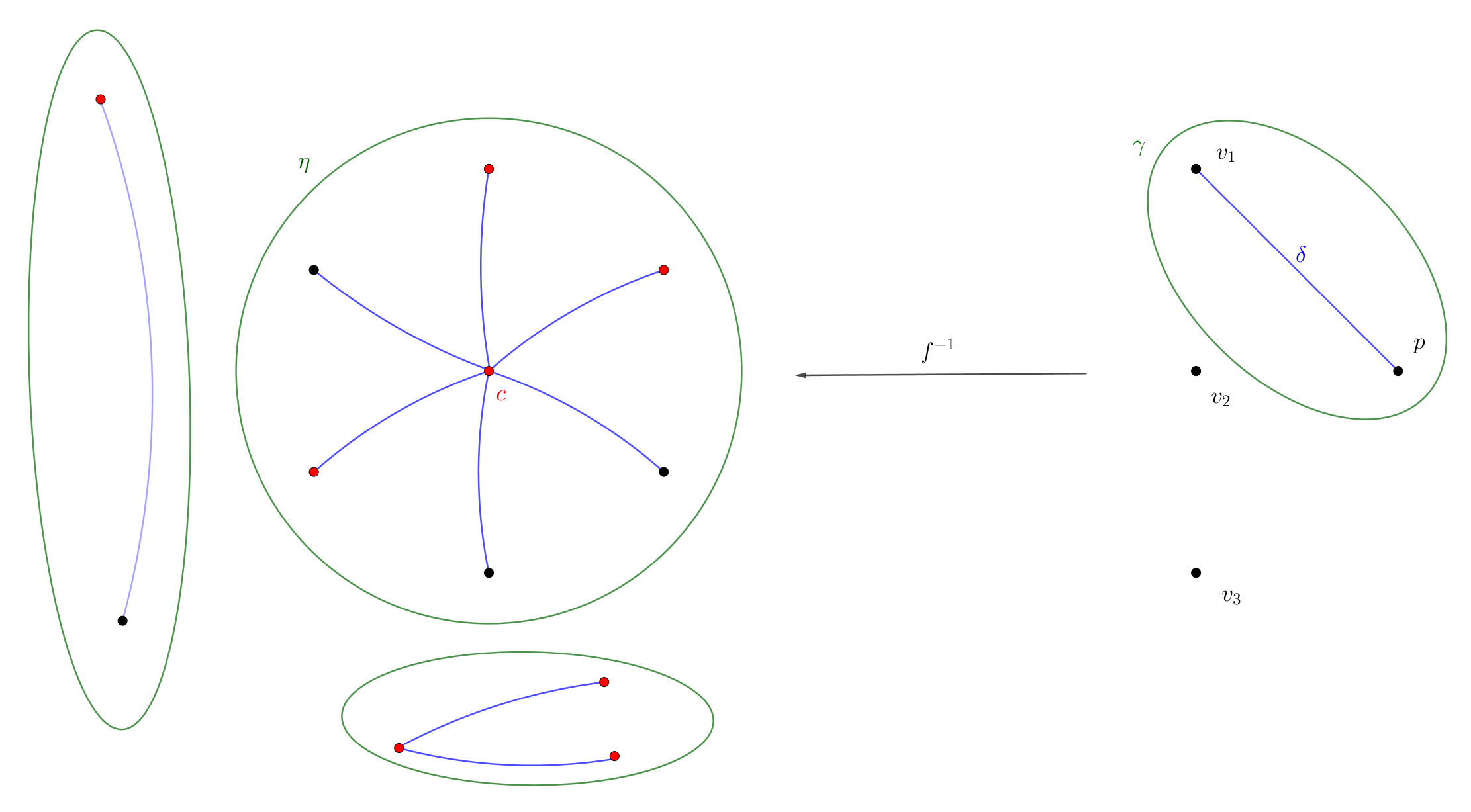}
\caption{A schematic picture illustrating the constancy of $\sigma_f$ in the proof of Proposition \ref{generalization?}.
The black points are the elements of the postcritical set ${\rm{P}}_f=\{v_1,v_2,v_3,p\}$ while the red points are not postcritical.  The simple closed curve $\gamma$ (in green) surrounds the fixed point $p$ and one of the critical values, say $v_1$; and $\delta$ (in blue) is an arc inside $\gamma$ connecting $v_1$ to $p$. Each connected component of 
$f^{-1}(\gamma)$ encloses a connected component of $f^{-1}(\delta)$. Each of the latter components is an arc or a ``star'' centered at a critical point above $v_1$ with its ends belonging to $f^{-1}(p)$. Due to the assumption on monodromy, there is a star component of $f^{-1}(\delta)$ (centered at the critical point $c$ in the picture) which is enclosed by a component of $\eta$ of $f^{-1}(\gamma)$ and has three of the postcritical points at its ends. There is only one other postcritical point, thus all components of $f^{-1}(\gamma)$ are peripheral  in  $\hat{\Bbb{C}}-{\rm{P}}_f$.
} 
\label{fig:tripod}
\end{figure}

\begin{remark}
The map \eqref{Kamalinezhad} is simultaneously Belyi and Lattès, and is induced by an endomorphism of the elliptic curve 
$y^2=x^3+1$. Indeed, if a Lattès map $f$ induced by a self-morphism (an affine transformation) of an elliptic curve $E$ is Belyi, then $E$ must have nontrivial automorphisms and hence is isomorphic to either $y^2=x^3+1$ or $y^2=x^3+x$. 
This is because $f:\hat{\Bbb{C}}\rightarrow\hat{\Bbb{C}}$ is semiconjugate to a morphism $E\rightarrow E$ via the action of a cyclic group $\Gamma$ of rotations around a base point for which $E\big/\Gamma\cong\hat{\Bbb{C}}$
\cite[Lemma 3.5]{MR2348953}. The number of critical values of $f$ is less than four only if $|\Gamma|>2$; hence $E$ admits an automorphism of order larger than two.  
\end{remark}

The final proposition of this section tries to mimic the treatment above of the example from  \cite[Appendix D]{MRArturo} 
in order to construct new examples of rational maps with constant pullback. 
\begin{proposition}\label{generalization?}
Let $h:\hat{\Bbb{C}}\rightarrow\hat{\Bbb{C}}$ be a Belyi map of degree $d$, $z_0$ be a regular value and 
$\rho:\pi_1\left(\hat{\Bbb{C}}-\{0,1,\infty\}\right)\rightarrow{\rm{Sym}}(h^{-1}(z_0))$ be the monodromy representation.
Denote the images of generators of  the fundamental group (determined by small positive loops around $0,1,\infty$) by $\theta_0,\theta_1,\theta_\infty$ (they satisfy $\theta_0\circ\theta_1\circ\theta_\infty={\rm{id}}$). 
Suppose there exist points $p,v_1,v_2,v_3$ of the fiber $h^{-1}(z_0)$ such that
\begin{itemize}
\item for any generator $\theta\in\{\theta_0,\theta_1,\theta_\infty\}$, at least three of $p,v_1,v_2,v_3$ are in the same cycle of the permutation $\rho(\theta)$;  
\item the cross-ratio of $p,v_1,v_2,v_3$ (in some ordering) is equal to $z_0$.
\end{itemize}
Then there exists a Möbius transformation $\alpha$ with the property that $f:=\alpha\circ h$ is a  PCF map with 
$|{\rm{P}}_f|=4$ whose pullback is constant. 
\end{proposition}

\begin{proof}
Due to the assumption on cross-ratios, after relabeling $p,v_1,v_2,v_3$ if necessary, 
there exists a Möbius transformation $\alpha$ with
$$
\alpha(z_0)=p,\quad\alpha(0)=v_1,\quad\alpha(\infty)=v_2,\quad\alpha(1)=v_3.
$$
(Notice that $T(z_0,0,\infty,1)=z_0$.)
The critical values of $\alpha\circ h$  are $v_1=\alpha(0),v_2=\alpha(\infty),v_3=\alpha(1)$. Since 
$p,v_1,v_2,v_3\in h^{-1}(z_0)$, 
the critical values $v_1,v_2,v_3$ of $f=\alpha\circ h$ are mapped $p=\alpha(z_0)$ which is a fixed point.
We see that $f$ is PCF with ${\rm{P}}_f=\{p,v_1,v_2,v_3\}$. 
Invoking \cite[Theorem 5.1]{MR3456156}, to prove the constancy of $\sigma_f$, it suffices to show that for any simple closed curve  $\gamma$ that encloses precisely two points of ${\rm{P}}_f$ the connected components of $f^{-1}(\gamma)$
are peripheral in $\hat{\Bbb{C}}-{\rm{P}}_f$.
The claim follows from the  argument presented in \cite[Appendix D]{MRArturo}: without any loss of generality, we may assume that $p,v_1$ are inside $\gamma$ and the other two postcritical points $v_2,v_3$ are outside it.  Consider an arc $\delta$ connecting $v_1$ to $p$ in the interior of $\gamma$. Its preimage $f^{-1}(\delta)$ 
consists of certain ``stars'' centered at critical points in 
$f^{-1}(v_1)$ with their ends being at points of $f^{-1}(p)$, or certain arcs that connect a regular point in $f^{-1}(v_1)$ to a point from $f^{-1}(p)$. Each of them is surrounded by a connected component of $f^{-1}(\gamma)$; see Figure \ref{fig:tripod}. 
That connected component is peripheral in $\hat{\Bbb{C}}-{\rm{P}}_f$ unless the component of 
$f^{-1}(\delta)$ inside it has exactly two points from ${\rm{P}}_f$ at its ends. 
But due to the assumption on monodromy, there is a critical point $c$ above $v_1$ whose corresponding connected component of $f^{-1}(\delta)$ has three members of ${\rm{P}}_f=\{p,v_1,v_2,v_3\}$ at its ends. So the connected component of $f^{-1}(\gamma)$ enclosing it, denoted by $\eta$ in Figure \ref{fig:tripod}, is peripheral. As for the rest of connected components of $f^{-1}(\gamma)$, the number of points of ${\rm{P}}_f$ inside them is zero or one, so they are peripheral.
\end{proof}

\begin{remark}
It is worthwhile to point out the power of the geometric argument presented in \cite[Appendix D]{MRArturo}. 
As in the proof above, that argument shows that the pullback is constant. But then, due to Proposition \ref{Belyi}, the fibers of the rational map are in a sense conformally related. Checking this directly can be tedious as observed in Example \ref{Arturo} where we algebraically verified that the fibers of \eqref{Arturo example} are conformally equivalent. 
\end{remark}

\begin{remark}
The hypothesis of Proposition \ref{Belyi} on the monodromy holds for the Belyi map \eqref{Kamalinezhad}. 
In a numbering of the edges of its dessin from Figure \ref{fig:dessin}, the permutation representation is  
$$
\theta_0=\left(1\quad 2\quad 3\right),\quad\quad\theta_1=\left(1\quad 3\quad 4\right)\quad\quad 
\theta_\infty=\left(2\quad 4\quad 3\right);
$$
observe that each of the permutations has three of $1,2,3,4$ in a cycle.  
\end{remark}

\subsection{Bicritical Thurston maps}
\begin{table}[hbt!]
\begin{tabular}{|p{8cm}|}
\hline
\vspace{1mm}\hfil
\xymatrixcolsep{3pc}\xymatrix
{\underset{v_1}{\bullet}\ar[r] & \bullet\ar@(dr,ur) & \underset{v_2}{\bullet}\ar[r] & \bullet\ar@(dr,ur)}\\
\hline
\vspace{1mm}\hfil
\xymatrixcolsep{3pc}\xymatrix
{\underset{v_1}{\bullet}\ar[r] & \bullet \ar@/^/[r] &\bullet \ar@/^/[l]& \ar[l]\underset{v_2}{\bullet}}\vspace{1mm}
\\
\hline
\end{tabular}
\vspace{2mm}
\caption{Possible functional graphs for the postcritical dynamics of a bicritical map with constant pullback; see the proof of Proposition \ref{bicritical}.}
\label{Tab:bicritical}
\end{table}
After discussing Thurston maps with three critical values in the previous subsection, we now focus on Thurston maps 
$f:S^2\rightarrow S^2$ with exactly two critical values. A simple application of Riemann-Hurwitz implies that the number of critical points should be two as well, i.e. $f$ is bicritical. 
\begin{proposition}\label{bicritical}
For bicritical Thurston maps with at least four postcritical points, the pullback is never constant. 
\end{proposition}

\begin{proof}
Assume the contrary: $f$ is a Thurston map of degree $d$ with $|{\rm{P}}_f|\geq 4$ and exactly two critical points for which $\sigma_f$ is constant. 
By Corollary \ref{main corollary}, one should have $|{\rm{P}}_f|\leq 2|{\rm{V}}_f|=4$; thus the equality is achieved. 
Proposition \ref{combinatorics} then implies that each critical value is strictly preperiodic and after one iteration lands on a periodic cycle. There are two possibilities:
\begin{itemize}
\item the images of critical values are two distinct fixed points;
\item the images of critical values form a $2$-cycle;
\end{itemize}
see Table \ref{Tab:bicritical}. 
One may directly check that $\sigma_f$ is not constant in either of these cases by a geometric curve lifting argument or by invoking \cite{MRParry1}. Instead, we shall employ an algebraic approach which involves an explicit computation of 
the associated Hurwitz correspondence $\mathcal{R}_f:\mathcal{M}_{0,{\rm{P}}_f}\rightrightarrows\mathcal{M}_{0,{\rm{P}}_f}$ (cf. \S2.1) from the diagram below 
\begin{equation}\label{correspondence'}
\xymatrix{\mathcal{T}\left(S^2,{\rm{P}}_f\right)\ar[dd]_{\pi} \ar[rd]^{\omega_f}  \ar[rr]^{\sigma_f}& &\mathcal{T}\left(S^2,{\rm{P}}_f\right)\ar[dd]^{\pi}\\
&\mathcal{W}_f \ar[ld]^{\rho_2}\ar[rd]_{\rho_1}&\\
\mathcal{M}_{0,{\rm{P}}_f}  & &\mathcal{M}_{0,{\rm{P}}_f}.}
\end{equation}
Since ${\rm{P}}_f$ is of size four, the moduli space $\mathcal{M}_{0,{\rm{P}}_f}$ can be identified with $\Bbb{C}-\{0,1\}$, and the Teichm\"uller space $\mathcal{T}\left(S^2,{\rm{P}}_f\right)$ with its universal cover, the upper half-plane $\Bbb{H}$. Suppose $\sigma_f([\phi])=[\phi']$. We pick representatives 
$\phi:\left(S^2,{\rm{P}}_f\right)\rightarrow \left(\hat{\Bbb{C}},\phi\left({\rm{P}}_f\right)\right)$
and 
$\phi':\left(S^2,{\rm{P}}_f\right)\rightarrow \left(\hat{\Bbb{C}},\phi'\left({\rm{P}}_f\right)\right)$
that take $v_1,v_2,f(v_1)$ to respectively $\infty,0,1$. Then $t:=\phi(f(v_2))$ and $t':=\phi'(f(v_2))$ provide coordinates for the versions of  $\mathcal{M}_{0,{\rm{P}}_f}\cong\Bbb{C}-\{0,1\}$ that appear on the left and right column of \eqref{correspondence'} respectively. 
Points of $\mathcal{W}_f$ are determined by rational maps $g$ that are Hurwitz equivalent to $f$: 
\begin{equation}\label{diagram'''}
\xymatrixcolsep{4pc}\xymatrix{\left(S^2,{\rm{P}}_f\right)\ar[r]^{\phi'}\ar[d]_f &\left(\hat{\Bbb{C}},\phi'\left({\rm{P}}_f\right)\right)\ar[d]^g\\ 
\left(S^2,{\rm{P}}_f\right)\ar[r]_\phi &\left(\hat{\Bbb{C}},\phi\left({\rm{P}}_f\right)\right).}
\end{equation}
The critical values of $g$ are $\infty,0$; and $g(\infty)=1$, $g(0)=t$. 
A normal form for such rational maps is
\begin{equation}\label{auxiliary5}
g(z)=\left(\frac{z+x}{z+y}\right)^d
\end{equation}
where $x,y\in\Bbb{C}$ satisfy $t=\left(\frac{x}{y}\right)^d$.  
In the case of the first portrait in Table \ref{Tab:bicritical}, $f(v_1)$ and $f(v_2)$ should be fixed points whereas in the other case $f(v_1)$ is mapped to $f(v_2)$ and vice versa. 
Chasing Diagram \eqref{diagram'''}, we are dealing with one of the following:
\begin{itemize}
\item $g(1)=1$ and $g(t')=t=\left(\frac{x}{y}\right)^d$; that is, $\frac{1+x}{1+y}$ and 
$\frac{\frac{t'+x}{t'+y}}{\frac{x}{y}}$ are $d^{\rm{th}}$ roots of unity;
\item $g(t')=1$ and $g(1)=t=\left(\frac{x}{y}\right)^d$; that is,   $\frac{t'+x}{t'+y}$ and 
$\frac{\frac{1+x}{1+y}}{\frac{x}{y}}$ are $d^{\rm{th}}$  roots of unity.
\end{itemize}
We denote the $d^{\rm{th}}$ roots of unity in each of the cases above by $\lambda$ and $\lambda'$; they are invariants of the Hurwitz class of $f$ (compare with \cite[Example 2.10]{MR3107522}).
In view of the above and $t,t'\in\Bbb{C}-\{0,1\}$, we conclude that $\mathcal{W}_f$ is an affine curve $C_{d,\lambda,\lambda'}$ which is defined as follows in the cases above:
\footnotesize
\begin{itemize}
\item $C_{d,\lambda,\lambda'}:=\left\{(x,y,t')\in\Bbb{C}^3\,\big|\, 1+x=\lambda(1+y),y(t'+x)=\lambda'x(t'+y)
\text{ where }t'\neq 0,1\text{ and }x,y\neq 0\text{ and }\left(\frac{x}{y}\right)^d\neq 1\right\}$;
\item $C_{d,\lambda,\lambda'}:=\left\{(x,y,t')\in\Bbb{C}^3\,\big|\, t'+x=\lambda(t'+y),y(1+x)=\lambda'x(1+y)
\text{ where }t'\neq 0,1\text{ and }x,y\neq 0\text{ and }\left(\frac{x}{y}\right)^d\neq 1\right\}$.
\end{itemize}
\normalsize
In conclusion, the correspondence associated with $f$ can be written as  
\begin{equation}\label{bicritical correspondence}
\xymatrixcolsep{7pc}\xymatrix{\Bbb{H}\cong\mathcal{T}\left(S^2,{\rm{P}}_f\right)\ar[dd]_\pi \ar[rd]^{\omega_f}  \ar[rr]^{\sigma_f}& &\mathcal{T}\left(S^2,{\rm{P}}_f\right)\cong\Bbb{H}\ar[dd]^\pi\\
&C_{d,\lambda,\lambda'} \ar[ld]^{ \rho_2: (x,y,t')\mapsto t=\left(\frac{x}{y}\right)^d}\ar[rd]_{\rho_1:(x,y,t')\mapsto t'\phantom{ssssss}}&\\
\Bbb{C}-\{0,1\}\cong\mathcal{M}_{0,{\rm{P}}_f} & &\mathcal{M}_{0,{\rm{P}}_f}\cong\Bbb{C}-\{0,1\}.}
\end{equation}
Notice that in both situations one should have $\lambda,\lambda'\neq 1$ for the map \eqref{auxiliary5}  to be of degree $d$
as otherwise the defining equations of $C_{d,\lambda,\lambda'}$  imply $x=y$. 
For  $\lambda,\lambda'\neq 1$, one can easily use the description of $C_{d,\lambda,\lambda'}$ provided before
to  check that in the commutative diagram \eqref{bicritical correspondence} $\rho_1$ is generically two-to-one; hence $\sigma_f$ is not constant. 
\end{proof}

\subsection{Cubic Thurston maps}
\begin{table}[hbt!]
\begin{tabular}{|p{6.5cm}|p{6cm}|}
\hline
\vspace{1mm}
\xymatrixcolsep{3pc}\xymatrix
{\underset{v_3}{\bullet}\ar[r] & \underset{t}{\bullet}\ar@(dr,ur)& \underset{v_1}{\bullet}\ar@(dr,ur) & \underset{v_2}{\bullet}\ar@(dr,ur) }
&\vspace{1mm}
\xymatrixcolsep{3pc}\xymatrix
{\underset{v_3}{\bullet}\ar[r] & \underset{t}{\bullet}\ar@(dr,ur)& \underset{v_1}{\bullet}\ar@/^/[r]  & \underset{v_2}{\bullet}\ar@/^/[l] }
\\
\hline
\vspace{1mm}
\xymatrixcolsep{3pc}\xymatrix
{\underset{v_3}{\bullet}\ar[r] & \underset{t}{\bullet}\ar@/^/[r]& \underset{v_1}{\bullet}\ar@/^/[l] & \underset{v_2}{\bullet}\ar@(dr,ur) }
&\vspace{1mm}\xymatrixcolsep{3pc}\xymatrix
{\underset{v_3}{\bullet}\ar[r] & \underset{t}{\bullet}\ar@/^/[r]& \underset{v_1}{\bullet}\ar@/^/[r] & \underset{v_2}{\bullet}\ar@/^1pc/[ll] }
\\
\hline
\vspace{1mm}
\xymatrixcolsep{3pc}\xymatrix
{\underset{v_3}{\bullet}\ar[r] & \underset{v_1}{\bullet}\ar[r]& \underset{t}{\bullet}\ar@(dr,ur) & \underset{v_2}{\bullet}\ar@(dr,ur) }
&\vspace{1mm}\xymatrixcolsep{3pc}\xymatrix
{\underset{v_3}{\bullet}\ar[r] & \underset{v_1}{\bullet}\ar[r]& \underset{t}{\bullet}\ar@/^/[r]& \underset{v_2}{\bullet}\ar@/^/[l]}
\\
\hline
\multicolumn{2}{|c|}{\xymatrixcolsep{3pc}\xymatrix
{\underset{v_3}{\bullet}\ar[r] & \underset{v_1}{\bullet}\ar[r]& \underset{v_2}{\bullet}\ar[r]&\underset{t}{\bullet}\ar@(dr,ur)}
}
\rule{0pt}{3ex}\\
\hline
\end{tabular}
\vspace{2mm}
\caption{In the proof of Proposition \ref{cubic}, we first narrow down possibilities for the postcritical dynamics  of a Thurston map $f$ with $\deg f=3$ and constant $\sigma_f$ to the portraits shown in this table. There are four postcritical points: the critical values $v_1,v_2,v_3$ and another point denoted by $t$ here. The critical point above $v_3$ is totally ramified while those above $v_1,v_2$ are nondegenerate. As illustrated here, there are seven possible postcritical portraits up to swapping $v_1$ and $v_2$.}
\label{Tab:cubic}
\end{table}
Utilizing the results obtained so far, we shall prove Theorem \ref{bicritical or cubic} in this final subsection. 
\begin{proposition}\label{cubic}
For Thurston maps of degree three with at least four postcritical points, the pullback is never constant.
\end{proposition}

\begin{proof}
Let $f$ be a Thurston map of degree three with $|{\rm{P}}_f|\geq 4$ for which $\sigma_f$ is constant.
First, suppose all critical points of $f$ are simple. If not all critical values are periodic, then one may pick a critical value $v$ whose preimage is disjoint from ${\rm{P}}_f$. The critical point above $v$ is simple and now Corollary \ref{ell2 corollary} implies that $\sigma_f$ is not constant. Thus all four critical values are periodic. Invoking Corollary \ref{periodic bound},  $\sigma_f$ is constant only if ${\rm{P}}_f={\rm{V}}_f$. But then $|{\rm{P}}_f|=4$ and $f$ is a NET map of degree three. It is known that the pullback for such a map cannot be constant \cite[Theorem 10.12]{MR2958932}, hence a contradiction.\\
\indent
Next, assume that $f$ admits a critical point at which the local degree is three. If there are two such points, then $f$ is bicritical and we have already seen in Proposition \ref{bicritical} that the pullback cannot be constant for such maps. So suppose the rest of the critical points are simple. Thus $f$ has a point of multiplicity three and two points of multiplicity two. We denote the critical values for the latter by $v_1$ and $v_2$ and the critical value for the former by $v_3$. Applying Corollary \ref{periodic bound} again, it is not possible for all of them to be periodic. Thus at least one of them is outside $f({\rm{P}}_f)$ as otherwise ${\rm{P}}_f=f({\rm{P}}_f)$ and every postcritical point is periodic. 
On the other hand, $v_1$ and $v_2$ must lie in $f({\rm{P}}_f)$ due to Corollary \ref{ell2 corollary}. 
We conclude that $v_3\notin f({\rm{P}}_f)$ while $v_1,v_2\in f({\rm{P}}_f)$. 
So if, as in the proof of Proposition \ref{combinatorics}, we consider the functional graph $\mathcal{G}_f$ of $f\restriction_{{\rm{P}}_f}:{\rm{P}}_f\rightarrow {\rm{P}}_f$, the indegrees of $v_1$ and $v_2$ are positive while that of $v_3$ is zero; and any other vertex is in ${\rm{P}}_f-{\rm{V}}_f$ so its indegree is at least two. We thus have:
\begin{equation*}
\begin{split}
|{\rm{P}}_f|&=\sum_{p\in {\rm{P}}_f}\deg^-(p)=\deg^-(v_1)+\deg^-(v_2)+\deg^-(v_3)+\sum_{p\in {\rm{P}}_f-{\rm{V}}_f}\deg^-(p)\\
&\geq 1+1+0+2(|{\rm{P}}_f|-3)=2|{\rm{P}}_f|-4.
\end{split}
\end{equation*}
But $|{\rm{P}}_f|\geq 4$; so we should have equality: aside from $v_1,v_2,v_3$, there is only one other postcritical point which we denote by $t$. As vertices of $\mathcal{G}_f$, their indegrees are $1,1,0$ and $2$ respectively. Possible functional graphs $\mathcal{G}_f$ (up to interchanging $v_1$ and $v_2$) are illustrated in Table \ref{Tab:cubic}.
One can directly check that the pullback is not constant in any of these cases by writing down the Hurwitz correspondence. A more elegant method is to apply Proposition \ref{Belyi} to the branched cover 
$f:\left(S^2,{\rm{P}}_f\right)\rightarrow\left(S^2,{\rm{P}}_f\right)$
which has precisely three critical values and satisfies $|{\rm{P}}_f\cap f^{-1}({\rm{V}}_f)|=2$: the pullback map $\sigma_f$ is constant only if $f$ admits a nontrivial deck transformation. But clearly no branched covering $f:S^2\rightarrow S^2$ of degree three with three critical values can admit a nontrivial deck transformation.
\end{proof}

\begin{proof}[Proof of Theorem \ref{bicritical or cubic}]
The result follows from Propositions \ref{bicritical} and \ref{cubic}.
\end{proof}

\noindent
\textbf{Acknowledgment.} The  author is grateful to Sarah Koch for many insightful conversations. I also would like to thank Bill Floyd, Walter Parry and Arturo Saenz for helpful comments and informing me of \cite{MRParry1} and \cite{MRParry2}.

\bibliographystyle{alpha}
\bibliography{bibcomprehensive}

\newcommand{\etalchar}[1]{$^{#1}$}
\begin{thebibliography}{FKK{\etalchar{+}}17}

\bibitem[BEK20]{MR4100123}
X.~Buff, A.~Epstein, and S.~Koch.
\newblock Eigenvalues of the {T}hurston operator.
\newblock {\em J. Topol.}, 13(3):969--1002, 2020.

\bibitem[BEKP09]{MR2508269}
X.~Buff, A.~Epstein, S.~Koch, and K.~M. Pilgrim.
\newblock On {T}hurston's pullback map.
\newblock In {\em Complex dynamics}, pages 561--583. A K Peters, Wellesley, MA,
  2009.

\bibitem[Bel83]{MR697314}
G.~V. Belyi.
\newblock On extensions of the maximal cyclotomic field having a given
  classical {G}alois group.
\newblock {\em J. Reine Angew. Math.}, 341:147--156, 1983.

\bibitem[Bel02]{MR1913596}
G.~V. Belyi.
\newblock A new proof of the three-point theorem.
\newblock {\em Mat. Sb.}, 193(3):21--24, 2002.

\bibitem[CFPP12]{MR2958932}
J.~W. Cannon, W.~J. Floyd, W.~R. Parry, and K.~M. Pilgrim.
\newblock Nearly {E}uclidean {T}hurston maps.
\newblock {\em Conform. Geom. Dyn.}, 16:209--255, 2012.

\bibitem[Cui]{MRCui}
G.~Cui.
\newblock Rational maps with constant {T}hurston map.
\newblock Conference of Complex Analysis of China (Huaqiao University 2014)
  \url{http://www.math.ac.cn/kyry/cgz/201501/W020150121410977178324.pdf}.

\bibitem[DH93]{MR1251582}
A.~Douady and J.~H. Hubbard.
\newblock A proof of {T}hurston's topological characterization of rational
  functions.
\newblock {\em Acta Math.}, 171(2):263--297, 1993.

\bibitem[Eps]{MREpstein}
A.~L. Epstein.
\newblock Transversality in holomorphic dynamics, manuscript available on
  \url{https://homepages.warwick.ac.uk/~mases/Transversality.pdf}.
\newblock pages 1--39.

\bibitem[FKK{\etalchar{+}}17]{MR3692494}
W.~J. Floyd, G.~Kelsey, S.~Koch, R.~Lodge, W.~R. Parry, K.~M. Pilgrim, and
  E.~A. Saenz.
\newblock Origami, affine maps, and complex dynamics.
\newblock {\em Arnold Math. J.}, 3(3):365--395, 2017.

\bibitem[FPS]{MRParry1}
W.~J. Floyd, W.~R. Parry, and E.~A. Saenz.
\newblock $\sigma_f$ constant implies that {${\rm{P}}_f$} almost equals
  {${\rm{V}}_f$}.
\newblock Unpublished.

\bibitem[GGD12]{MR2895884}
E.~Girondo and G.~Gonz\'{a}lez-Diez.
\newblock {\em Introduction to compact {R}iemann surfaces and dessins
  d'enfants}, volume~79 of {\em London Mathematical Society Student Texts}.
\newblock Cambridge University Press, Cambridge, 2012.

\bibitem[HK17]{MR3668372}
E.~Hironaka and S.~Koch.
\newblock A disconnected deformation space of rational maps.
\newblock {\em J. Mod. Dyn.}, 11:409--423, 2017.

\bibitem[Hub06]{MR2245223}
J.~H. Hubbard.
\newblock {\em Teichm\"{u}ller theory and applications to geometry, topology,
  and dynamics. {V}ol. 1}.
\newblock Matrix Editions, Ithaca, NY, 2006.
\newblock Teichm\"{u}ller theory, With contributions by Adrien Douady, William
  Dunbar, Roland Roeder, Sylvain Bonnot, David Brown, Allen Hatcher, Chris
  Hruska and Sudeb Mitra, With forewords by William Thurston and Clifford
  Earle.

\bibitem[Hub16]{MR3675959}
J.~H. Hubbard.
\newblock {\em Teichm\"{u}ller theory and applications to geometry, topology,
  and dynamics. {V}ol. 2}.
\newblock Matrix Editions, Ithaca, NY, 2016.
\newblock Surface homeomorphisms and rational functions.

\bibitem[Kam13]{MR3431861}
A.~Kamalinezhad.
\newblock On the geometrization of the absolute {G}alois group.
\newblock {\em Fundam. Prikl. Mat.}, 18(6):145--159, 2013.

\bibitem[Koc]{MRKoch}
S.~Koch.
\newblock Unpublished.

\bibitem[Koc13]{MR3107522}
S.~Koch.
\newblock Teichm\"{u}ller theory and critically finite endomorphisms.
\newblock {\em Adv. Math.}, 248:573--617, 2013.

\bibitem[KPS16]{MR3456156}
S.~Koch, K.~M. Pilgrim, and N.~Selinger.
\newblock Pullback invariants of {T}hurston maps.
\newblock {\em Trans. Amer. Math. Soc.}, 368(7):4621--4655, 2016.

\bibitem[LZ04]{MR2036721}
S.~K. Lando and A.~K. Zvonkin.
\newblock {\em Graphs on surfaces and their applications}, volume 141 of {\em
  Encyclopaedia of Mathematical Sciences}.
\newblock Springer-Verlag, Berlin, 2004.
\newblock With an appendix by Don B. Zagier, Low-Dimensional Topology, II.

\bibitem[Mil06]{MR2348953}
J.~W. Milnor.
\newblock On {L}att\`es maps.
\newblock In {\em Dynamics on the {R}iemann sphere}, pages 9--43. Eur. Math.
  Soc., Z\"urich, 2006.

\bibitem[Par]{MRParry2}
W.~R. Parry.
\newblock Notes on polynomials with constant pullback maps.
\newblock Unpublished.

\bibitem[Pil22]{MR4472056}
K.~M. Pilgrim.
\newblock On the pullback relation on curves induced by a {T}hurston map.
\newblock In {\em In the tradition of {T}hurston. {II}. {G}eometry and groups},
  pages 385--399. Springer, 2022.

\bibitem[Sae12]{MRArturo}
E.~A. Saenz.
\newblock On nearly euclidean thurston maps.
\newblock {\em Ph.D. Thesis, Virginia Tech}, 2012.

\bibitem[Sae19]{MR3994789}
E.~A. Saenz.
\newblock On {NET} maps: examples and nonexistence results.
\newblock {\em Conform. Geom. Dyn.}, 23:147--163, 2019.

\bibitem[Woo06]{MR2266994}
Melanie~Matchett Wood.
\newblock Belyi-extending maps and the {G}alois action on dessins d'enfants.
\newblock {\em Publ. Res. Inst. Math. Sci.}, 42(3):721--737, 2006.

\end{thebibliography}

\end{document}